\documentclass{amsart}
\usepackage[left=1in,top=1in,right=1in,bottom=1in]{geometry} %Margins
\usepackage{amsmath,amsfonts,amsthm,amssymb}
\usepackage{graphicx,enumerate}
\usepackage{pdfpages}

\newcommand{\R}{\mathbb{R}}

\newcommand{\Z}{\mathbb{Z}}
\newcommand{\N}{\mathbb{N}}
\newcommand{\s}{\mathcal{S}}

\def\v{\varphi}
\def\r{\rho}
\def\n{\mathcal{N}}
\def\f{\mathcal{F}}
\def\a{\alpha}
\def\G{\Gamma}
\def\b{\beta}

\newcommand{\ie}{{\em i.$\,$e.$\!$} }
\newcommand{\e}{\varepsilon}
\newcommand{\g}{\gamma}
\newcommand{\ep}{\epsilon}
\newcommand{\M}{\mathcal{M}}
\newcommand{\NF}{\mathcal{NF}}
\def\ll{\Lambda}
\def\beq{\begin{equation}}
\def\eeq{\end{equation}}
\newcommand\numberthis{\addtocounter{equation}{1}\tag{\theequation}}
\newtheorem{theorem}{Theorem}[section]

\newtheorem{lemma}[theorem]{Lemma}

\newtheorem*{claim*}{Claim}
\newtheorem{proposition}[theorem]{Proposition}
\newtheorem{observation}[theorem]{Observation}

\theoremstyle{definition}

\theoremstyle{remark}

\DeclareMathOperator{\supp}{supp}

\numberwithin{equation}{section}

\allowdisplaybreaks

\title[Boundedness of the (Sub)Bilinear maximal function]{The Boundedness of the (Sub)Bilinear Maximal Function along ``non-flat" smooth curves}

\author{Alejandra Gaitan and Victor Lie}

\date{\today}
\address{Alejandra Gaitan: Department of Mathematics\\Purdue University\\
West Lafayette, Indiana\\47906}

\email{ygaitanm@purdue.edu}

\address{Victor Lie: Department of Mathematics\\Purdue University\\
West Lafayette, Indiana\\47906}

\email{vlie@purdue.edu}

\address{And Institute of Mathematics of the
Romanian Academy, Bucharest, RO 70700, P.O. Box 1-764, Romania.}

\thanks{The second author was supported by the National Science Foundation under Grant No. DMS-1500958. The final revision of the paper before publication was performed while the second author was supported by the National Science Foundation under Grant No. DMS-1900801.}

\keywords{Bilinear Hilbert transform and maximal operators along ``non-flat" curves, (non)stationary phase, Littlewood-Paley theory, shifted maximal function, shifted square function.}

\begin{document}

\begin{abstract}
Let $\NF$ be the class of smooth non-flat curves near the origin and near infinity introduced in \cite{lie2015boundedness} and let $\gamma\in\NF$. We show - via a unifying approach relative to the corresponding bilinear Hilbert transform $H_{\Gamma}$ - that the (sub)bilinear maximal function along curves $\Gamma=(t,-\gamma(t))$ defined as $$M_\Gamma(f,g)(x):=\sup\limits_{\e>0} \frac{1}{2\e} \int_{-\e}^\e |f(x-t)g(x+\gamma(t))|dt$$
is bounded from $L^p(\R)\times L^{q}(\R)\to L^r(\R)$ for all $p, q$ and $r$ H\"older indices, \ie $\frac{1}{p}+\frac{1}{q}=\frac{1}{r}$, with $1<p,\,q\leq\infty$ and $1\le r\leq\infty$. This is the maximal boundedness range for $M_{\Gamma}$, that is, our result is \emph{sharp}.
\end{abstract}

\maketitle

%%%%%%%%%%%%%%%%%%%%%%%%%%%%%%%%%%%%%%%%%%%%%%%%%%%%%%%%%%%%%%%%%%%%%%%%%%%%%%%%
%%%%%%%%%%%%%%%%%%%%%%%%%%%%%%%%%%%%%%%%%%%%%%%%%%%%%%%%%%%%%%%%%%%%%%%%%%%%%%%%
%%%% Introduction
%%%%%%%%%%%%%%%%%%%%%%%%%%%%%%%%%%%%%%%%%%%%%%%%%%%%%%%%%%%%%%%%%%%%%%%%%%%%%%%%
\section{Introduction}

In this paper we study the boundedness properties of the bilinear\footnote{Throughout the paper we allow this slight abuse by referring to our maximal operator as ``bilinear", though of course, strictly speaking, this is a sub-linear operator in each of the two inputs.} maximal function along a properly chosen class of curves. That is, given $\Gamma=(t,-\gamma(t))$ where $\gamma$ is a suitable smooth, non-flat curve near zero and near infinity, we ask about the boundedness properties of the bilinear maximal function along $\Gamma$ defined by
\begin{align}
 M_\Gamma&:\s(\R)\times\s(\R)\to L^{\infty}(\R) \nonumber\\
 M_\Gamma(f,g)(x)&:=\sup\limits_{\e>0} \frac{1}{2\e} \int_{-\e}^\e |f(x-t)g(x+\gamma(t))|dt.\label{definition}
\end{align}

Seen as a forerunner to \cite{lvUA3}, this paper presents a short proof of the maximal boundedness range for $M_{\Gamma}$. This proof along with the corresponding approach to the singular integral version $H_{\Gamma}$ in \cite{lie2015boundedness} and \cite{lie2015triangle} emerge as constituent parts of a whole. Together with the work in \cite{lie2019unifiedap} and its follow-up \cite{lvUA3}, this should be regarded as part of a larger enterprise to provide a unified perspective on several central themes in harmonic analysis that deal with the boundedeness properties of various classes of singular integral and maximal operators.

\subsection{Historical background and motivation}

The problem regarding the boundedness properties of the maximal bilinear operator $M_{\G}$ and of its singular integral analogue defined by
\beq\label{nhilb}
H_{\Gamma}(f,g)(x):= \textrm{p.v.}\int_{\R} f(x-t)\,g(x+\g(t))\frac{dt}{t}
\eeq
has a long history being initially motivated in areas such as ergodic theory - in relation with almost everywhere convergence of bilinear averages (\cite{Bou}) or the $L^p$-norm convergence of (non-)conventional bilinear averages (see e.g. \cite{Fu}, \cite{HKr}) as well as PDE area - in relation to commutators involving differential operators (\cite{Cal1}, \cite{calderon1978commutators}).

A brief account\footnote{Our bibliography here is not exhaustive, listing only the references that are closest and most relevant to our concise study.} of the development of this subject within harmonic analysis is given by
\begin{itemize}
\item in the zero-curvature (or flat) case, \emph{i.e.} when $\g(t)=a t$ for some $a\in\R\setminus\{-1\}$, the problem of providing bounds for the bilinear Hilbert transform $H_{\Gamma}$ was raised by A. Calder\'on in his study of the Cauchy transform along Lipschitz curves, \cite{Cal}. This was approached by M. Lacey and C. Thiele in  \cite{lt1} and \cite{lt2} where they proved that $H_{\Gamma}$ obeys $L^p\times L^q \to L^r$ bounds\footnote{It is still an open problem if this range is optimal.} for $\frac{1}{p}+\frac{1}{q}=\frac{1}{r}$ with $1<p,\,q\leq \infty$ and $\frac{2}{3}<r<\infty$. The analogous result for the operator $M_{\Gamma}$ was proved by Lacey in \cite{la3}.

\item  in the nonzero-curvature (or non-flat) case, \emph{e.g.} when $\g''$ is nonzero near origin and infinity,  the problem of providing bounds for $H_{\Gamma}$ was first addressed by X. Li (\cite{li}), in the special case $\Gamma(t) = (t, t^d)$, $2 \leq d \in \mathbb N$, by proving that $H_{\Gamma}$ obeys $L^{2}\times L^{2} \to L^{1}$ bounds. His proof relies on the concept of $\sigma$-uniformity introduced in \cite{cltt} and is inspired by  Gowers's work in \cite{gowers}.

    In \cite{lie2015boundedness} and \cite{lie2015triangle} the second author of the present paper proved the maximal range up to end-points for $H_{\G}$ where here $\g$ belongs to a suitable class of curves $\n\f$ that includes in particular any generalized (Laurent) real polynomial with no linear term.\footnote{\emph{I.e.}, any expression of the form $\g(t):=\sum_{j=1}^d a_j t^{\a_j}$ with $\{a_j\}_{j=1}^{d}\subset\R$, $\{\a_j\}_{j=1}^{d}\subset\R\setminus\{1\}$ and $d\in\N$. Here, we use the following convention: if $\a,\,t\in \R$ we let $t^{\a}$ stand for either $|t|^{\a}$ or $\textrm{sgn}\, (t)\, |t|^{\a}$. In a follow up paper, \cite{AG}, the first author extends the present result (and its singular integral analogue in \cite{lie2015boundedness} and \cite{lie2015triangle}) to the case in which one allows the curve $\g$ to be a generalized polynomial but with the linear term included.} Besides the general character of the class of curves $\n\f$, a novelty of \cite{lie2015boundedness}  is that it correctly identifies a scale-type decay relative to the size of the phase of the multiplier. The proof of this result resides on wave-packet analysis (Gabor frames), time-frequency discretization techniques, and orthogonality methods.

    Later, elaborating on the ideas and techniques in both \cite{li} and \cite{lie2015boundedness}, the authors in \cite{liXiao2016uniform} prove for both $H_{\Gamma}$  and $M_{\Gamma}$ the expected H\"older range in the case $\g(t)$ a standard polynomial with no linear term with bounds that are uniform in the polynomial's coefficients.
\end{itemize}

For more on the historical evolution of our problem and further connections with other mathematical subjects one is invited to consult \cite{lie2019unifiedap}.

\subsection{The main result}
 In what follows, we refer to \cite[Section 2]{lie2015boundedness} for the definition of the class $\NF$ of smooth ``non-flat'' curves near zero and infinity.

With this we can state the main theorem of our paper:

\begin{theorem}\label{main_theorem}
 Let $\Gamma=(t,-\gamma(t))$ be a curve such that\footnote{It is easy to notice that our main result extends to the class of curves $\NF^{C}$ that is defined to be the set of all curves $\g\,+\,C$ with $\g\in\NF$ and $C\in\R$, \textit{i.e.}, our result is closed under translation by constants of $\g$ with $\g\in\NF$.} $\gamma\in\NF$. Consider the bilinear maximal function defined by \eqref{definition}.

 Then $M_\Gamma$ extends boundedly from $L^p(\R)\times L^{q}(\R)$ into $L^r(\R)$ where the indices $p, q, r$ obey
	\begin{equation}\label{exponents10}
	\frac{1}{p}+\frac{1}{q}=\frac{1}{r},
	\end{equation}
	and\footnote{Observe that one gets trivially the desired bound for the triple of indices $(p,q,r)=(\infty,\infty,\infty)$ corresponding to the point $C$ in Figure \ref{triangle_fig}. That is why we will exclude this case from all our future reasonings.}
	\begin{equation}\label{exponents20}
	1<p\leq\infty\,,\, 1<q\leq\infty \mbox{ and } 1\leq r\leq \infty.
	\end{equation}
This result is sharp.
\end{theorem}

\begin{observation}
Let $\mathcal{P}_d$ be the class of all real polynomials of degree $d$ with no linear and constant terms. In \cite[Section 3]{liXiao2016uniform}, the authors show that for any $d\in\N$, $d\ge2$, there exists a polynomial
$P_d\in\mathcal{P}_d$ such that for $\Gamma=(t,P(t))$ one has that $M_\Gamma$ is unbounded whenever $p, q, r$ obey \eqref{exponents10} with $1<p, q<\infty$ and $r<\frac{d-1}{d}$.

However, for any $d\in\N$, $d\ge2$ we have that $\mathcal{P}_d\subset\NF$. Thus, in order for \eqref{exponents10} and \eqref{exponents20} to hold for any $\gamma\in\NF$ we must have $r\ge 1$, hence the claimed optimality of the range in our theorem above.
\end{observation}

\subsection{Main ideas; relevance}\label{MIR}

As already mentioned earlier, our approach of the maximal operator $M_{\G}$ is developed along a natural correspondence with the singular integral approach associated to the bilinear Hilbert transform $H_{\G}$. For the remaining part of this section, in order to make our reasoning transparent, we will keep our presentation at an \emph{informal level}.

The philosophy behind our approach relies on the following observation:

The bilinear Hilbert transform $H_{\G}$ can be written as
\begin{equation}\label{hilg}
H_{\G}(f,g)=\sum_{j\in\Z} H_j(f,g):=\sum_{j\in\Z} \int_{\R} f(x-t)\,g(x+\g(t)) \r_j(t)\,dt\,,
\end{equation}
where  here $\rho_j(t)=2^j\,\rho(2^j t)$ for a suitable $\rho\in C^\infty_0(\R)$ with $\supp \rho \subseteq\{t\in\R|\frac{1}{4}<|t|<1\}$ and obeying the \emph{mean zero condition}
\begin{equation}\label{mzh}
\int_{\R}\rho(t)dt=0\,.
\end{equation}

In contrast with this, assuming from now on wlog that $f,\,g\geq 0$, the bilinear maximal function $M_{\G}$ can be expressed as
\begin{equation}\label{mg}
M_{\G}(f,g)=\sup_{j\in\Z} M_j(f,g):=\sup_{j\in\Z} \int_{\R} f(x-t)\,g(x+\g(t))\underline{\r}_j(t)\,dt\,,
\end{equation}
where  as before $\underline{\rho}_j(t)=2^j\,\underline{\rho}(2^j t)$ for a suitable $\underline{\rho}\in C^\infty_0(\R)$ \emph{positive} with $\supp \underline{\rho} \subseteq\{t\in\R|\frac{1}{4}<|t|<1\}$ and obeying
\begin{equation}\label{nzh}
\int_{\R}\underline{\rho}(t)dt=1\,.
\end{equation}

\noindent \textsf{Remark 1} \emph{Thus, in a nutshell, one has
\begin{itemize}
\item $H_{\G}$ is a conditional $l^1$-sum of pieces $\{H_j\}_{j\in\Z}$ with the associated kernels $\{\r_j\}$ having mean zero;
\item $M_{\G}$ is an $l^{\infty}$-sum of pieces $\{M_j\}_{j\in\Z}$ with the associated kernels $\{\underline{\r}_j\}$ integrating to one.
\end{itemize}}

Once at this point, one can isolate the corresponding components $M_j$ (and $H_j$) and regard them as bilinear Fourier multiplier operators. In doing so, it becomes transparent that the phase oscillation in the integral definition of the multiplier will play a key role in the proof. Thus, following the strategy designed for $H_{\G}$ in \cite{lie2015boundedness}, we analyze the stationary points of the phase and decompose accordingly\footnote{Same type of decomposition holds for $H_j$.}
\begin{equation}\label{multiplierM}
 M_j=M_j^L+M_j^{H\not\Delta}+M_j^{H\Delta}\,,
\end{equation}
where
\begin{itemize}
\item $M_j^L$ is the $j-$low frequency component (essentially no phase oscillation);

\item $M_j^{H\not\Delta}$ is the $j-$high frequency component away from the stationary points region;

\item $M_j^{H\Delta}$ is the $j-$high frequency component along the stationary points region.
\end{itemize}

\noindent \textsf{Remark 2} \emph{As it turns out, only the first components $M_j^L$ and its Hilbert analogue $H_j^L$ are sensitive to the existent distinction
between \eqref{nzh} and \eqref{mzh}. Consequently, for the last two components one will be able to identify $M_j^{H\not\Delta}$ and $M_j^{H\Delta}$ with their analogue $H_j^{H\not\Delta}$ and $H_j^{H\Delta}$, respectively.}

From here on, the strategy follows the dichotomy present in \cite{lie2015boundedness}, \cite{lie2015triangle} and \cite{lie2019unifiedap}:

\begin{itemize}
\item the control over the low frequency component $\M^L_\Gamma:=\sup_j |M_j^L|$ is obtained via Taylor series expansions exploiting the lack of oscillation on the multiplier side; indeed, Theorem \ref{low_frequency_theorem} states that\footnote{Here $M$ stands for the standard Hardy-Littlewood maximal function.}
\begin{equation}\label{lf}
 |\M_\Gamma^L(f,g)(x)|\lesssim_\gamma Mf(x)Mg(x).
\end{equation}

\item the bounds for the high frequency component away from the stationary points region $\M^{H\not\Delta}_\Gamma:=\sup_{j} |M_j^{H\not\Delta}|$ rely on a further discretization combined with (non)-stationary principle in disguise - essentially a careful integration by parts procedure - and a novel shifted square/maximal function argument. As a byproduct of the latter, each of the elementary building blocks in the decomposition of $\M^{H\not\Delta}_\Gamma$ will be pointwise bounded by a product of shifted maximal functions (multiplied by a suitable decaying factor) thus mirroring \eqref{lf} (see relation \eqref{mvst}). The superposition of these pointwise estimates will provide us with the global control over $\M^{H\not\Delta}_\Gamma$.
     This is the content of Theorem \ref{far_theorem}.

\item the high frequency component along the stationary points region defined as $\M^{H\Delta}_\Gamma:=\sup_{j}| M_j^{H\Delta}|$ is of course the main term of our operator. After the linearization of the supremum, one decomposes the main term as
 \begin{equation}\label{decm}
 M_{j(x)}^{H\Delta}(f,g)(x)=\sum_{m\in\N}  M_{j(x),m}^{H\Delta}(f,g)(x)\,,
\end{equation}
where each $M_{j(x),m}^{H\Delta}(f,g)(x)$ has the phase of the multiplier oscillating at height $\approx2^m$.

Applying now Remark 2, for $j(x)=j$, one can identify $M_{j,m}^{H\Delta}(f,g)$ with the analogue $H_{j,m}^{H\Delta}(f,g)$ and thus use the key estimate obtained in \cite[Theorem 3]{lie2015boundedness} to get that there exists $\epsilon\in(0,1)$ such that, for any $j\in \Z$ and $m\in\N$, one has
\begin{equation}\label{l2cs}
 \|M^{H\Delta}_{j,m}(f,g)\|_{L^1}\lesssim_{\gamma} 2^{-\epsilon m}\|f\|_2\,\|g\|_{2}.
\end{equation}

At this point hinted by Remark 1 and the approach in \cite{lie2015boundedness}, one makes the simple observation
\begin{equation}\label{l2css}
 \|M^{H\Delta}_{j(x),m}(f,g)(x)\|_{L^1(dx)}\leq \sum_{j} \|M^{H\Delta}_{j,m}(f,g)\|_{L^1}\lesssim_{\gamma} 2^{-\epsilon m}\|f\|_2\,\|g\|_{2}\,,
\end{equation}
where in the last line we used Cauchy-Schwarz and the almost orthogonality of the inputs along the sequence $\{M^{H\Delta}_{j,m}(f,g)\}_{j\in\Z}$. This takes care of the bound $L^2\times L^2 \to L^1$.

The complete boundedness range, stated in Theorem \ref{high_frequency_diagonal theorem} relies in part on the techniques developed in \cite{lie2015triangle} along with (shifted or generalized) square function arguments.
\end{itemize}

The above description is part of a more general, philosophical approach - see \cite{lie2019unifiedap} - of treating simultaneously and in a unitary fashion both the singular - here $H_{\Gamma}$ - and its maximal variant - here $M_{\G}$.

Finally, this paper is meant as a preface to the significantly more complex study in \cite{lvUA3}, in which, completing the unification of the three themes introduced in \cite{lie2019unifiedap}, we will develop a unified approach for the boundedness of $H_{\G}$ and $M_{\G}$ in the case in which $\Gamma=(t,-\gamma(x,t))$ - thus allowing an $x-$dependence of $\g$ - where here $\g$ is a suitable non-degenerate curve that is smooth and doubling in $t$ but only \emph{measurable} in $x$.

\section{Notation}

Without lost of generality, from now on throughout the paper, we will assume that $f$ and $g$ are non-negative functions. For transparency, we will mostly follow the notations and conventions as in \cite{lie2015boundedness}.

For example, given any $\phi\in\s(\R)$, $j,\,l\in\Z$ and $k\in\N$, we set
\begin{align}
\phi_j(x)&:=\phi\big(\frac{x}{2^j}\big),\label{notation2}\\
\tilde\phi_{l}(x)&:=x^l\phi(x),\label{notation21}\\
\phi^{(k)}(x)&:=\left(\frac{d}{dx}\right)^k\phi(x),\label{notation22}\\
\tilde\phi_{l,j}(x)&:=\tilde\phi_l\big(\frac{x}{2^j}\big). \label{notation23}
\end{align}

If $f\in \s(\R)$ we denote the Fourier transform of $f$ with $\hat{f}$, where\footnote{In our later reasonings we will often ignore the constant $\frac{1}{\sqrt{2\pi}}$.}
$$\hat{f}(\xi):=\frac{1}{\sqrt{2\pi}}\int_{\R} f(x)\,e^{-i x \xi}\,d x\,,\:\:\:\:\:\:\xi\in\R\,$$
and the inverse Fourier transform of $f$ with $\check{f}$, where
$$\check{f}(\xi):=\frac{1}{\sqrt{2\pi}}\int_{\R} f(x)\,e^{i x \xi}\,d x\,,\:\:\:\:\:\:\xi\in\R\,.$$

We denote by $M$ the standard Hardy-Littlewood maximal operator defined on $L^1_\text{loc}(\R)$ as
\[Mf(x):=\sup\limits_{\substack{I\ni x\\ I\text{ interval}}}\frac{1}{|I|}\int_I |f(x)|dx.\]

For $t\in\R$ and $\v\in C^\infty_0$ with $\v(0)=0$ we define the $t$-\emph{shifted square functions} with respect to $\v$ by
\begin{equation}\label{shif_max}
S_t^{\v}f(x):=\Big(\sum\limits_{j\in\Z}\Big|(f* (2^j\check{\v}(2^j\cdot)))\big(x-\frac{t}{2^j} \big)\Big|^2\Big)^\frac{1}{2}\,,
\end{equation}
and the $t$-\emph{shifted maximal function} by\footnote{As expected, in the maximal function case, the mean zero condition of $\v$ becomes irrelevant and one can replace the original conditions imposed on $\v$ by merely $\check{\v}\in \s(\R)$ any (normalized) function with $\check{\v}\geq 0$ and $\|\check{\v}\|_{L^1(\R)}=1$. Given this, in what follows we no longer specify the dependence of our maximal function on $\v$.}
\begin{equation}\label{shif_maxM}
M^{(t)}f(x):=\sup\limits_{j\in\Z}\Big|(f* (2^j\check{\v}(2^j\cdot)))\big(x-\frac{t}{2^j} \big)\Big|\,.
\end{equation}

If $\frac{1}{4} <|t|<1$, $\v\in C^\infty_0$ with $\v(0)=0$, $n\in\Z $ and $\g\in\NF$ we define the $(n,\g_t)$-\emph{shifted square functions} with respect to $\v$ by

\begin{equation}\label{shif_maxgama}
S_{n,\g_t}^{\v}f(x):=\Big(\sum\limits_{j\in\Z}\Big|\Big(f*\big\{\frac{1}{\tfrac{2^{-j}\gamma'(2^{-j})}{2^n}} \check{\v} (\frac{\cdot}{\tfrac{2^{-j}\gamma'(2^{-j})}{2^n}})\big\}\Big) (x+\gamma(2^{-j}t))\Big|^2\Big)^\frac{1}{2}\,,
\end{equation}
and the $(n,\g_t)$-\emph{shifted maximal function} by
\begin{equation}\label{shif_maxgamaM}
M^{(2^n)}_{\g_t}f(x):=\sup\limits_{j\in\Z}\Big|\Big(f*\big\{\frac{1}{\tfrac{2^{-j}\gamma'(2^{-j})}{2^n}} \check{\v} (\frac{\cdot}{\tfrac{2^{-j}\gamma'(2^{-j})}{2^n}})\big\}\Big) (x+\gamma(2^{-j}t))\Big|\,.
\end{equation}

Throughout the paper $p^*=\min\{p,p'\}$ where $p'$ denotes de H\"older conjugate of $p$.

Also we set $a_+:=\max\{a,0\}$.

For $A, B>0$, we say that $A\lesssim B$ if there is $C>0$ such that $A\le CB$.  We say that $A\approx B$ if $A\lesssim B$ and $B\lesssim A$. Finally, if $d$ is a real parameter, we write $A\lesssim_d B$ if there exists $C=C(d)>0$ such that $A\le CB$ with the obvious correspondence for $A\approx_d B$.

\section{Preliminaries: The bilinear maximal function as a multiplier operator; Reduction to the main component}

In this section, we first reshape the maximal operator in a convenient form adapted to Fourier analytic methods followed by a careful analysis of the associated multiplier. As a result, we will be able to decompose our initial operator in three components: the first two of them - considered as ``error terms" - will be treated in the present section, while the remaining one - \emph{i.e} the main component - will be left for the next sections.

Focusing now on the main subject, we record the following simple observation: since we deal with a positive integral operator, it is enough to study our maximal function with the supremum ranging over dyadic numbers, \textit{i.e.}, letting  $\e\sim 2^{j+1}$ with $j\in\Z$, we have that
\begin{equation}\label{discretization}
M_{\Gamma}(f,g)(x)\approx\sup\limits_{j\in\Z}\frac{1}{2^{j+1}}\int\limits_{2^j\le|t|\le2^{j+1}} f(x-t)\,g(x+\gamma(t))\,dt.
\end{equation}
Based on the properties of the curves $\g\in\NF$ we have that\footnote{The condition discussed below will be needed when discussing the boundedness of our operator $M_{\Gamma}(f,g)$ in the regime $L^{\infty}\times L^q$ for $1<q\leq \infty$ - see Subsection 5.2.}
\begin{align}
\exists\:\:\textrm{(possibly large)}&\:\: j_0\in\N\:\:\textrm{depending only on}\:\g\:\textrm{such that} \label{threshold}\\
\textrm{for any nonzero}&\:|t|\leq 2^{-j_0}\:\textrm{or}\:|t|\geq 2^{j_0}\:\textrm{one has}\nonumber\\
|\g(t)|>0,\:\: &|\g'(t)|>0,\:\:\textrm{and}\:\:|\g''(t)|>0\:.\label{threshold1}
\end{align}
With this, based on \eqref{discretization}, we notice that
\begin{align}
M_{\Gamma}(f,g)(x)&\lesssim \sup\limits_{|j|>j_0}\frac{1}{2^{j+1}}\int\limits_{2^j\le|t|\le2^{j+1}} f(x-t)\,g(x+\gamma(t))\,dt\,\label{discretizationsI}\\
&+\,\sum_{|j|\leq j_0}\frac{1}{2^{j+1}}\int\limits_{2^j\le|t|\le2^{j+1}} f(x-t)\,g(x+\gamma(t))\,dt\,.\label{discretizationsII}
\end{align}

Now for the latter term, using triangle and Minkowski's inequality, we trivially have
\begin{align*}
&\:\:\:\:\left\|\sum_{|j|\leq j_0} \frac{1}{2^{j+1}} \int\limits_{2^j\le|t|\le2^{j+1}} f(x-t)\,g(x+\gamma(t))\,dt\right\|_{L^r(dx)}\\
&\leq  \sum_{|j|\leq j_0} \int\limits_{2^j\le|t|\le2^{j+1}} \|f(x-t)\,g(x+\gamma(t))\|_{L^r(dx)}\,\frac{dt}{2^{j+1}}\\
&\lesssim_{p,r} \,j_0\,\|f\|_p\,\|g\|_{q}\,,
\end{align*}
with $p,\,q,\,r$ satisfying \eqref{exponents10} and \eqref{exponents20}.

Thus, with a slight notational abuse, we will assume from now on that
\begin{equation}\label{discretizationc}
M_{\Gamma}(f,g)(x)=\sup\limits_{{j\in\Z}\atop{|j|>j_0}}\frac{1}{2^{j+1}}\int\limits_{2^j\le|t|\le2^{j+1}} f(x-t)\,g(x+\gamma(t))\,dt.
\end{equation}

Let now $\underline{\rho}$ be a nonnegative, even, $C^\infty_0(\R)$ function with $\supp \underline{\rho} \subseteq\{t\in\R|\frac{1}{4}<|t|<1\}$ and
\begin{equation}\label{intones}
\int\underline{\rho}(t)dt=1\,.
\end{equation}

Set $\underline{\rho}_j(t)=2^j\underline{\rho}(2^j t)$ (with $j\in\Z$).

With this, standard reasonings show that $M_{\Gamma}\approx \M_{\Gamma}$, where here
\begin{equation}
\M_{\Gamma}(f,g)(x):=\sup\limits_{{j\in\Z}\atop{|j|>j_0}}\M_{\Gamma,j}(f,g)(x):=\sup\limits_{{j\in\Z}\atop{|j|>j_0}}\int_\R f(x-t)g(x+\gamma(t))\underline{\rho}_j(t)dt\,.
\end{equation}

Turning our attention on the Fourier side, we have that
\begin{equation}\label{M_multiplier}
 \M_\Gamma(f,g)(x)=\sup\limits_{{j\in\Z}\atop{|j|>j_0}} \int_\R\int_\R \hat{f}(\xi)\hat{g}(\eta)m_j(\xi,\eta)e^{i\xi x}e^{i\eta x}d\xi d\eta\,,
\end{equation}
where the multiplier is given by
\begin{equation}\label{M_multiplierr}
 m_j(\xi,\eta):=\int_\R e^{-i\xi t}e^{i\eta \gamma(t)} \underline{\rho}_j(t)dt
=\int_\R e^{-i\frac{\xi}{2^j} t}e^{i\eta \gamma(\frac{t}{2^j})} \underline{\rho}(t)dt.
\end{equation}

Since the integrand in $m_{j}$ is highly oscillatory, the analysis of our multiplier relies fundamentally
on understanding the stationary points of the phase function
\begin{equation}\label{phmultiplier}
 \varphi_{\xi,\eta}(t):=-\frac{\xi}{2^{j}}t+\eta\gamma\big(\frac{t}{2^{j}}\big)\:.
\end{equation}

Thus, as in \cite{lie2015boundedness} and \cite{lie2019unifiedap}, it is natural to decompose the multiplier based on the size of the terms $\frac{\xi}{2^{j}}$ and $\frac{\gamma'(2^{-j})\eta}{2^{j}}$.
Let now $\phi$ be a positive even Schwartz function supported in $\{t:\frac{1}{2}\le|t|\le 2\}$ with
$$\sum_{n\in\Z}\phi_n(t)=1\:\:\: \textrm{for all}\:\:\:t\not=0\,.$$

Then, for every $j\in\Z\setminus [-j_0,\,j_0]$, we write
\begin{equation}
m_j=\sum_{m,n\in\Z} m_{j,m,n},
\end{equation}
where
\begin{equation}\label{mjmn}
m_{j,m,n}(\xi,\eta):=m_j(\xi,\eta)\,\phi(\frac{\xi}{2^{m+j}})\,\phi\Big(\frac{\gamma'(2^{-j})\eta}{2^{n+j}}\Big)\,.
\end{equation}

From the stationary phase principle  we have that the main contribution for the integrand in $m_j$ comes from the values near the stationary point(s). With this, we follow the approach in \cite{lie2015boundedness}, and split the analysis of our multiplier into three components, corresponding to the behavior of the phase in \eqref{phmultiplier}, as follows:

\begin{enumerate}
 \item[I)] \underline{\emph{Low frequency case}} - the phase function is essentially constant:
\begin{equation}\label{low_freq}
m_j^L=\sum_{(m,n)\in(\Z_-)^2} m_{j,m,n}.
\end{equation}
\item[II)] \underline{\emph{High frequency far from diagonal}} - high oscillation without stationary points:
\begin{equation}\label{high_far}
  m_j^{H\not\Delta}=\sum_{(m,n)\in\Z^2\setminus((\Z_-)^2\cup\Delta)}m_{j,m,n}.
\end{equation}
\item[III)] \underline{\emph{High frequency close to diagonal}} - high oscillation with present stationary points:
\begin{equation}\label{high-close}
 m_j^{H\Delta}=\sum_{(m,n)\in\Delta}m_{j,m,n}.
\end{equation}
\end{enumerate}
Here $\Delta=\{(m,n)\in\Z^2:m,n\ge0,|m-n|\le C(\gamma)\}$ with $C(\gamma)\ge 1$ a large constant depending only on $\gamma$.

With this, we have
\begin{equation}\label{multiplier}
 m_j=m_j^L+m_j^{H\not\Delta}+m_j^{H\Delta}.
\end{equation}

We end our preliminaries by transforming the maximal nature of our operator via a linearization procedure
\begin{equation}\label{multiplier1}
  \M_\Gamma(f,g)(x)\approx\int_\R\int_\R \hat{f}(\xi)\hat{g}(\eta)m_{j(x)}(\xi,\eta)e^{i\xi x}e^{i\eta x}d\xi d\eta
\end{equation}
where $j(x):\R\to\Z\setminus [-j_0,\,j_0]$ is a measurable function who assigns for each point $x\in\R$ a value $j\in \Z\setminus [-j_0,\,j_0]$ for which $\M_{\Gamma,j}(f,g)(x)$ is at least half of the value of $\M_\Gamma(f,g)(x)$.

%%%%%%%%%%%%%%%%%%%%%%%%%%%%%%%%%%%%%%%%%%%%%%%%%%%%%%%%%%%%%%%%%%%%%%%%%%%%%%%%
%%%%%%%%%%%%%%%%%%%%%%%%%%%%%%%%%%%%%%%%%%%%%%%%%%%%%%%%%%%%%%%%%%%%%%%%%%%%%%%%
%%%%%%%%%%%%%%%%%%%%%%%%%%%%%%%%%%%%%%%%%%%%%%%%%%%%%%%%%%%%%%%%%%%%%%%%%%%%%%%%
%%%% subsection - Low frequency terms
\subsection{Low frequency term}\label{low_frequency-subsection}
%% Pointwise bound with the HL Maximal function
As discussed in \cite{lie2019unifiedap}, the dichotomy between the singular integral behavior and its maximal version manifests precisely in the low frequency case: indeed, this is the only situation that requires the mean zero condition of the kernel $\frac{1}{t}$ in the bilinear Hilbert transform case which translates, after the standard discretization procedure, into a condition of the from \eqref{mzh} as opposed to \eqref{nzh} in the maximal case. Based on this, the bilinear Hilbert transform required similar techniques with those used to prove the Coifman-Meyer theorem in order to obtain the necessary bounds for the corresponding low frequency component (see \cite[Theorem 1]{lie2015boundedness}).

In the bilinear maximal function case, our function $\underline{\rho}$ only satisfies $\int\underline{\rho}(t)dt=1$ but the good news is that - recall Remark 1 in Section 1.3. -  we only need to control an $l^{\infty}(\Z)$-sum/norm (in $j$) as opposed to a conditional $l^1(\Z)$-sum in the bilinear Hilbert transform case.

With this being said, we have

\begin{theorem}\label{low_frequency_theorem}
Set
\begin{equation}\label{operator_low_freq}
 \M^L_\Gamma(f,g)(x):=\int_\R \int_\R \hat{f}(\xi)\hat{g}(\eta)m_{j(x)}^L(\xi,\eta)e^{i\xi x}e^{i\eta x}d\xi d\eta.
\end{equation}
Then, the following holds
\begin{equation}\label{pointwise_maximal}
 |\M_\Gamma^L(f,g)(x)|\lesssim_\gamma Mf(x)Mg(x).
\end{equation}
Furthermore, for any $p, q$ and $r$ satisfying \eqref{exponents10} and \eqref{exponents20} we have
\begin{equation}\label{pointwise_maximal_bound}
 \|\M_\Gamma^L(f,g)\|_r\lesssim_{\gamma,p,q}\|f\|_p\|g\|_{q}.
\end{equation}
\end{theorem}

%% Proof
\begin{proof}
Fix $m,n \in\Z_-$. From the definition of $\phi$ we have that $\big|\frac{\xi}{2^{j(x)}}\big|, \big|\frac{\gamma'(2^{-j(x)})\eta}{2^{j(x)}}\big|\lesssim 1$, and by property (3)  as part of the definition of $\NF$ in \cite{lie2015boundedness}, we have
\[
\sup\limits_{t\in\supp\underline{\rho}}\Big|\eta\gamma\Big(\frac{t}{2^{j(x)}} \Big)\Big|<C_\gamma,
\]
where $C_\gamma$ is a positive constant that is allowed to change from line to line.

Recalling now \eqref{M_multiplierr} - \eqref{mjmn}, we develop the phase in \eqref{phmultiplier} in a Taylor series
\begin{align}
 m_{j(x),m,n}(\xi,\eta)&=\Bigg(\int_\R \sum_{k\in\N} \frac{(-i\frac{\xi}{2^{j(x)}}t)^k}{k!}\sum_{l\in\N}\frac{(i\eta\gamma(\frac{t}{2^{j(x)}}))^l}{l\,!}\underline{\rho}(t)dt\Bigg)\phi\Big(\frac{\xi}{2^{m+j(x)}} \Big)\phi\Big(\frac{\gamma'(2^{-j(x)})\eta}{2^{n+j(x)}}\Big)\nonumber\\
% %
&=\sum_{k,l\ge0}\frac{(-1)^k i^{k+l}}{k!\,l\,!}C_{k,l}
\Big(\frac{\xi}{2^{j(x)}}\Big)^k\phi\Big(\frac{\xi}{2^{m+j(x)}}\Big) \Big(\frac{\gamma'(2^{-j(x)})\eta}{2^{j(x)}}\Big)^l\phi\Big(\frac{\gamma'(2^{-j(x)})\eta}{2^{n+j(x)}}\Big),\label{taylor-series}
\end{align}
where
\[
 C_{k,l}=C_{j(x),k,l}:=\int_\R t^k\Big(\frac{\gamma(2^{-j(x)}t)}{2^{-j(x)}\gamma'(2^{-j(x)})}\Big)^l \underline{\rho}(t)dt.
\]
Since
\begin{equation}\label{constcont}
 |C_{k,l}|\lesssim C_\gamma^{k+l}
\end{equation}
uniformly in $x$, the sum \eqref{taylor-series} converges absolutely.

Set now
\begin{equation}\label{psi}
\psi(x):=\sum_{m<0}\phi(\frac{x}{2^m})\,,
\end{equation}
and notice that $\psi\in\s(\R)$ since $\psi=1-\sum_{m\ge0}\phi_m.$

Inserting \eqref{taylor-series} in \eqref{low_freq} and recalling \eqref{notation21}, we have
\begin{align}
 m^L_{j(x)}(\xi,\eta)&=\sum_{k,l\ge0} \frac{(-1)^k i^{k+l}}{k!\,l\,!}C_{k,l} \Big(\frac{\xi}{2^{j(x)}}\Big)^k\psi\Big(\frac{\xi}{2^{j(x)}}\Big) \Big(\frac{\gamma'(2^{-j(x)})\eta}{2^{j(x)}}\Big)^l
\psi\Big(\frac{\gamma'(2^{-j(x)})\eta}{2^{j(x)}}\Big),\\
&=\sum_{k,l\ge0} \frac{(-1)^k i^{k+l}}{k!\,l\,!}C_{k,l} \,\tilde\psi_k\Big(\frac{\xi}{2^{j(x)}}\Big) \tilde\psi_l\Big(\frac{\gamma'(2^{-j(x)})\eta}{2^{j(x)}}\Big).\label{lowTaylor}
\end{align}

Thus, putting together \eqref{lowTaylor} and \eqref{operator_low_freq} and recalling \eqref{notation22}, we have
\begin{align}
\M_\Gamma^L(f,g)(x)&=\sum_{k,l\ge0} \frac{(-1)^k i^{k+l}}{k!\,l\,!}C_{k,l} \Big(\int_\R \hat{f}(\xi)\tilde\psi_k\Big(\frac{\xi}{2^{j(x)}}\Big)e^{i\xi x}d\xi\Big) \Big(\int_\R \hat{g}(\eta)\tilde\psi_l\Big(\frac{\gamma'(2^{-j(x)})\eta}{2^{j(x)}}\Big) e^{i\eta x}d\eta\Big)\nonumber\\
&\approx\sum_{k,l\ge0} \frac{(-1)^k }{k!\,l\,!}C_{k,l} \Big(f*\Big\{\frac{1}{2^{-j(x)}}\check{\psi}^{(k)}\Big(\frac{\cdot}{2^{-j(x)}}\Big)\Big\}\Big)(x)\label{low_frequency}\\
&\hspace{2cm}\times \Big(g*\Big\{\frac{1}{2^{-j(x)}\gamma'(2^{-j(x)})}\check{\psi}^{(l)} \Big(\frac{\cdot}{2^{-j(x)}\gamma'(2^{-j(x)})}\Big)\Big\}\Big)(x).\nonumber
\end{align}

We can construct the function $\phi$ such that for every $k\in\N$, the function $\check{\psi}^{(k)}$ has an integrable radially decreasing majorant $\Phi_k$ with $\|\Phi_k\|_1\lesssim C^k$ for some constant $C>0$. Thus, using a classical result in \cite{stein1993harmonic}, one has
\begin{align}
\Big|\Big(f*\Big\{\frac{1}{2^{-j(x)}}\check{\psi}^{(k)}\Big(\frac{\cdot}{2^{-j(x)}}\Big)\Big\}\Big)(x)\Big|&\le \sup_{t>0} \Big|\Big(f*\Big\{\frac{1}{t}\check{\psi}^{(k)}\Big(\frac{\cdot}{t}\Big)\Big\}\Big)(x)\Big|\nonumber\\
&\le \|\Phi_k\|_1 Mf(x)\nonumber\\
&\lesssim C^k Mf(x).\label{Mf}
\end{align}

Analogously,
\begin{equation}\label{Mg}
\Big|\Big(g*\Big\{\frac{1}{2^{-j(x)}\gamma'(2^{-j(x)})}\check{\psi}^{(l)} \Big(\frac{\cdot}{2^{-j(x)}\gamma'(2^{-j(x)})}\Big)\Big\}\Big)(x)\Big|\lesssim C^l Mg(x).
\end{equation}
Putting together \eqref{low_frequency}, \eqref{Mf} and \eqref{Mg} we have
\begin{align*}
|\M_\Gamma^L(f,g)(x)|&\lesssim \sum\limits_{k,l\ge0}\frac{1}{k!\,l\,!}|C_{k,l}|C^{k+l}Mf(x)Mg(x)\\
&\lesssim Mf(x)Mg(x)\sum\limits_{k,l\ge0}\frac{1}{k!\,l\,!}C_\gamma^{k+l} C^{k+l}\\
&\lesssim_\gamma Mf(x)Mg(x),\numberthis\label{pointwise_result}
\end{align*}
where here we used \eqref{constcont}.

Therefore, \eqref{pointwise_maximal} holds.

Taking now $p,\, q,\, r$ satisfying \eqref{exponents10} and \eqref{exponents20}, H\"older’s inequality and the standard $L^p$ strong type estimate for the classical Hardy-Littlewood maximal function imply
 \begin{align*}
\|\M_\Gamma^L(f,g)\|_r&\lesssim_\gamma\|Mf Mg\|_r\lesssim_{p,q} \|Mf\|_p\|Mg\|_q\lesssim_{p,q} \|f\|_p \|g\|_q,
 \end{align*}
which concludes our proof.
\end{proof}
%%%%%%% END Proof - HL Maximal and Holder exponents

%%%%%%%%%%%%%%%%%%%%%%%%%%%%%%%%%%%%%%%%%%%%%%%%%%%%%%%%%%%%%%%%%%%%%%%%%%%%%%%%
%%%%%%%%%%%%%%%%%%%%%%%%%%%%%%%%%%%%%%%%%%%%%%%%%%%%%%%%%%%%%%%%%%%%%%%%%%%%%%%%
%%%%%%%%% High frequency terms
\subsection{High frequency term far from diagonal}
% Off-diagonal terms

In this section we discuss the second (off-diagonal) term in our decomposition of $m_{j(x)}$, that is, $m_{j(x)}^{H\not\Delta}$.

Our main focus, will be to prove the following

\begin{theorem}\label{far_theorem}
 Set
 \begin{equation}\label{operator_far}
 \M^{H\not\Delta}(f,g)(x):=\int_\R \int_\R \hat{f}(\xi)\hat{g}(\eta)m_{j(x)}^{H\not\Delta}(\xi,\eta)e^{i\xi x}e^{i\eta
x}d\xi d\eta,
 \end{equation}
 and for $(m,n)\in\Z^2\setminus((\Z_-)^2\cup\Delta)$ let
  \begin{equation}\label{operator_far}
 \M^{H\not\Delta}_{m,n}(f,g)(x):=\int_\R \int_\R \hat{f}(\xi)\hat{g}(\eta)m_{j(x),m,n}(\xi,\eta)e^{i\xi x}e^{i\eta
 	x}d\xi d\eta.
 \end{equation}
 Then, for any $p, q, r$ satisfying \eqref{exponents10} and  \eqref{exponents20}, the following holds
 \begin{itemize}
 	\item If $(m,n)\in\Z^2\setminus((\Z_-)^2\cup\Delta)$ then, recalling \eqref{shif_maxM} and \eqref{shif_maxgamaM}, one has the pointwise estimate
  \begin{equation}\label{mvst}
|\M^{H\not\Delta}_{m,n}(f,g)(x)|\lesssim \frac{1}{2^{\max\{m,n\}}}\,\int\limits_{\{\frac{1}{4} <|t|<1 \}} M^{(2^m t)}f(x)\,M^{(2^n)}_{\g_t} g(x)\,dt\:\:\:\:\:\:\:\:\forall\:x\in\R\,,
\end{equation}
which further implies
 	\begin{equation}
 	\|\M^{H\not\Delta}_{m,n}(f,g)\|_r\lesssim_{\gamma,p,q}\frac{1}{2^{\max\{m,n\}}}(m_++10) (n_++10)\|f\|_p\|g\|_{q}\,.\label{far_bound}
 	\end{equation}

\item Deduce thus that the global component obeys
 	\begin{equation}
 	\|\M_\Gamma^{H\not\Delta}(f,g)\|_r\lesssim_{\gamma,p,q}\|f\|_p\|g\|_q.\label{far_thm_bound}
 	\end{equation}
 \end{itemize}
\end{theorem}

We first state two lemmas which will be used in the proof of the above statement.

The first lemma contains two items: point i) represents Lemma 4.8 proved in \cite{lie2015triangle} and can be found in the math literature under various modified forms - see \textit{e.g.} \cite{MS} and \cite{stein1993harmonic}. The second point ii) can be proved with similar techniques to those used for i) and relies in a key fashion on the properties of the class of curves $\NF$. The second lemma mirrors the first one and addresses the maximal function analogue. The case of $M^{(2^n t)}$ was also referred to and proved within Proposition 42 in \cite{lie2019unifiedap}.
%%%%%% Shifted square functions lemma
\begin{lemma}\label{shifted_square_funct}
Let $\frac{1}{4} <|t|<1$, $\v\in C^\infty_0$ with $\v(0)=0$, $n\in\Z $ and $\g\in\NF$.

\noindent i) Recall the definition of $S_t^{\v}$ given by \eqref{shif_max} in the Notation section.

Then, for any $1<p<\infty$, one has -- uniformly in $t$ -- that
\[
  \|S_{2^n t}^{\v}f\|_p\lesssim_{\v,p} (1+n_{+})^{\frac{2}{p^*}-1}\|f\|_p.
\]

\noindent ii) Similarly, if $S_{n,\g_t}^{\v}$ stands for the shifted square functions defined in \eqref{shif_maxgama} then, for any $1<p<\infty$, one has -- uniformly in $t$ -- that
\begin{equation}\label{sup_displaced}
\|S_{n,\g_t}^{\v}f\|_p\lesssim_{\g,\v,p}(1+n_{+})^{\frac{2}{p^*}-1}\|f\|_p\:.
 \end{equation}
\end{lemma}

\begin{lemma}\label{shifted_max_funct}
With the same assumptions as above, and appealing to \eqref{shif_maxM} and \eqref{shif_maxgamaM} in the Notation section we have
that
\begin{equation}\label{sup_displaced}
\|M^{(2^n t)}f\|_p,\:\|M^{(2^n)}_{\g_t}f\|_p\lesssim_{\g,\v,p}(1+n_{+})^{\frac{1}{p}}\|f\|_p\:,
 \end{equation}
holds uniformly in $t$ and for all $1<p\leq\infty$.
\end{lemma}

With these we are now ready to present

\begin{proof}[Proof of Theorem \ref{far_theorem}]

Fix a pair $(m,n)\in\Z^2\setminus((\Z_-)^2\cup\Delta)$. As in the proof of \cite[Claim 1]{lie2015boundedness}, integrating by parts, we write
\[
 m_{j(x),m,n}=A_{j(x),m,n}(\xi,\eta)+B_{j(x),m,n}(\xi,\eta),
\]
where
\begin{equation}
 A_{j(x),m,n}(\xi,\eta):=\Big(\int_\R e^{-i\frac{\xi}{2^{j(x)}}t} e^{i\eta\gamma(\frac{t}{2^{j(x)}})} \frac{i\
\underline{\rho}'(t)}{-\frac{\xi}{2^{j(x)}}+\frac{\eta}{2^{j(x)}}\gamma'\big(\frac{t}{2^{j(x)}}\big)}dt\Big)
\phi\Big(\frac{\xi}{2^{m+j(x)}} \Big) \phi\Big(\frac{\gamma'(2^{-j(x)})\eta}{2^{n+j(x)}}\Big),
\end{equation}
and
\begin{equation}
 B_{j(x),m,n}(\xi,\eta):= \Big(\int\limits_\R e^{-i\frac{\xi}{2^{j(x)}}t}e^{i\eta\gamma\big(\frac{t}{2^{j(x)}}\big)}
\frac{-i\frac{\eta}{2^{2j(x)}}\gamma''\big(\frac{t}{2^{j(x)}}\big)} {\Big({-\frac{\xi}{2^{j(x)}}+\frac{\eta}{2^{j(x)}}\gamma'\big(\frac{t}{2^{j(x)}}\big)}\Big)^2} \underline{\rho}(t)dt\Big)
\phi\Big(\frac{\xi}{2^{m+j(x)}}\Big)\phi\Big(\frac{\gamma'(2^{-j(x)})\eta}{2^{n+j(x)}}\Big).
\end{equation}

%%%%%%%%%%  Case 1 (proof of Corollary 1)
\item \textsf{Case 1}: $m-n\ge C_\gamma>>1$
$\newline$

By (3) in the definition of $\NF$, we have that
\[
\tfrac{1}{-\tfrac{\xi}{2^{j(x)}}+\tfrac{\eta}{2^{j(x)}}\gamma'\big(\frac{t}{2^{j(x)}}\big)}=-\tfrac{1}{2^m}\frac{1}{\frac{\xi}{2^{m+j(x)}}} \sum_{l\in\N} \tfrac{1}{2^{l(m-n)}}\Bigg(\frac{\frac{\gamma'(2^{-j(x)})\eta}{2^{n+j(x)}}}{\frac
{\xi}{2^{m+j(x)}}}\Bigg)^l (Q'(t)+Q'_{j(x)}(t))^l.
\]
Set
\[
 A_{j(x),m,n}=\sum_{l\in\N} A_{j(x),m,n,l},
\]
with
\begin{align*}
 A_{j(x),m,n,l}(\xi,\eta)&=-\frac{1}{2^m}\frac{1}{2^{l(m-n)}} \left(\int_\R e^{-i\frac{\xi}{2^{j(x)}}t} e^{i\eta\gamma\big(\frac{t}{2^{j(x)}}\big)} i\,\underline{\rho}'(t) (Q'(t)+Q_{j(x)}(t))^l dt\right)\\
&\hspace{2.5cm}\times
\tilde{\phi}_{-l-1}\Big(\frac{\xi}{2^{m+j(x)}}\Big)\tilde{\phi}_{l} \Big(\frac{\gamma'(2^{-j(x)})\eta}{2^{n+j(x)}}\Big)
\end{align*}
where here $\tilde{\phi}_{l}$, $l\in\Z$, is smooth, compactly supported away from the origin, and obeying
\[
\|\tilde{\phi}_{l}\|_{C^\beta}\lesssim\beta!\,(|l|+\b)^{\b}\,
C^{|l|+\b}
\]
for $\b\in\N$ and some fixed $C>0$.

We recall now the fact that, from the properties of $\g\in\NF$ one has that $\|Q_j\|_{C^N}\le a_j$ with $a_j\to 0$ as
$j\to\infty$ and thus, since $|j(x)|\geq j_0$, one can choose $j_0\in\N$ large enough so that $\|Q_{j(x)}\|_{C^N}\leq \|Q\|_{C^N}\:\:\textrm{a.e.}\:x\in\R$.

 Defining the operator $\Lambda^A$ as
\[
\Lambda^A_{m,n,l}(f,g)(x):=\int_\R\int_\R \hat{f}(\xi)\hat{g}(\eta) A_{j(x),m,n,l}(\xi,\eta)e^{i\xi x}e^{i\eta x}d\xi d \eta,
\]
we have
\begin{align}
|\Lambda^A_{m,n,l}(f,g)(x)|&=\Big|\frac{1}{2^m}\frac{1}{2^{l(m-n)}} \int_\R (f*\check{\tilde{\phi}}_{-l-1,m+j(x)})\big(x-\frac{2^mt}{2^{m+j(x)}}\big)\label{sqv}\\
&\hspace{.1cm}\times \big(g*\big\{\tilde{\phi}_l\big(\tfrac{2^{-j(x)}\gamma'(2^{-j(x)})}{2^n}\,\cdot\big)\big\} \check{\phantom{1}}\big)\big(x+\gamma(2^{-j(x)}t)\big) \,\underline{\rho}'(t)\,(Q'(t)+Q'_{j(x)}(t))^l\, dt\Big|\nonumber\\
&\lesssim \frac{1}{2^m}\Big(\frac{2\|Q'\|_\infty}{2^{m-n}}\Big)^l \int\limits_{\{t\in\R:\frac{1}{4} <|t|<1 \}}
\sup_{j\in\Z}\Big|(f*\check{\tilde{\phi}}_{-l-1,j})\big(x-\frac{2^mt}{2^j}\big)\Big| \nonumber\\
&\hspace{.5cm}\times
\sup_{j\in\Z}\Big|\big(g*\big\{\frac{1}{\tfrac{2^{-j}\gamma'(2^{-j})}{2^n}} \check{\tilde{\phi}}_l (\frac{\cdot}{\tfrac{2^{-j}\gamma'(2^{-j})}{2^n}})\big\}\big) (x+\gamma(2^{-j}t))\,\Big|\,dt\nonumber\:.
\end{align}
Once at this point one can proceed in two ways:
\begin{itemize}
\item the simplest route is to apply the estimate
\begin{equation}\label{mv}
|\Lambda^A_{m,n,l}(f,g)(x)|\lesssim \frac{1}{2^m}\Big(\frac{C_{\g}}{2^{m-n}}\Big)^l \int\limits_{\{\frac{1}{4} <|t|<1 \}} M^{(2^m t)}f(x)\,M^{(2^n)}_{\g_t} g(x)\,dt\,.
\end{equation}
\item alternatively, one can appeal to a shifted square function argument, and write
\begin{equation}\label{sv}
|\Lambda^A_{m,n,l}(f,g)(x)|\lesssim \frac{1}{2^m}\Big(\frac{C_{\g}}{2^{m-n}}\Big)^l \left(\int\limits_{\{\frac{1}{4} <|t|<1 \}} S_{2^mt}^{\tilde{\phi}_{-l-1}}f(x)\,S_{n,\g_t}^{\tilde{\phi}_l}g(x)\,dt\right)\,.
\end{equation}
\end{itemize}

Let $\Lambda^A_{m,n}:=\sum_{l\in\N} \Lambda^A_{m,n,l}$. Using now Fubini, H\"older, Lemma \ref{shifted_max_funct} and \eqref{mv} (or alternatively, Lemma \ref{shifted_square_funct} and \eqref{sv}), one deduces
\begin{equation}\label{A1}
\|\Lambda^A_{m,n}(f,g)\|_r\lesssim_{\gamma,p,q}\frac{1}{2^m}(m_++10)(n_{+}+10)\,\|f\|_p\,\|g\|_q\:.
\end{equation}

Similarly, one gets the analogues of \eqref{A1} for the multiplier $B_{j,m,n}$, where in this latter case one uses that
$\frac{1}{2^{2j}}\gamma''(2^{-j}t)=2^{-j}\gamma'(2^{-j})(Q''(t)+Q_j''(t))\,.$
$\newline$

 %%%%%%%%%  Case 2
\item \textsf{Case 2}: $n-m>C_\g>>1$
$\newline$

As we did in Case 1, we write
\[
\frac{1}{-\frac{\xi}{2^{j(x)}}+\frac{\eta}{2^{j(x)}}
	\gamma'(\frac{t}{2^{j(x)}})}
=\frac{1}{2^{n}}\frac{1}{\frac{\gamma'(2^{-j(x)})\eta}{2^{n+j(x)}}}
\sum_{l\in\N}\frac{1}{2^{l(n-m)}} \Bigg(\frac{\frac{\xi}{2^{m+j(x)}}}
{\frac{\gamma'(2^{-j(x)})\eta}{2^{n+j(x)}}}\Bigg)^l\frac{1}{(Q'(t)+Q'_{j(x)}
	(t))^{l+1}}.
\]
Write
\[
\tilde{A}_{j(x),m,n}=\sum_{l\in\N}\tilde{A}_{j(x),m,n,l},
\]
where
\begin{align*}
\tilde{A}_{j(x),m,n,l}(\xi,\eta)&=\frac{1}{2^{n}}\frac{1}{2^{l(n-m)}}
\left(\int_\R
e^{-i\frac{\xi}{2^{j(x)}}t}e^{i\eta\gamma\big(\frac{t}{2^{j(x)}}\big)}i
\frac{\underline{\rho}'(t)}{(Q'(t)+Q'_{j(x)}(t))^{l+1}}dt\right) \\
&\hspace{2cm}\times
\tilde{\phi}_{l}\left(\frac{\xi}{2^{m+j(x)}}\right)\,\tilde{\phi}_{-l-1}\left(\frac{
	\gamma'(2^{-j(x)})\,\eta}{2^{n+j(x)}}\right).
\end{align*}

Letting $\Lambda^{\tilde{A}}_{m,n,l}$ be the operator with symbol
$\tilde{A}_{j(\cdot),m,n,l}$ one proceeds as before in order to get
\begin{equation*}
|\Lambda^{\tilde{A}}_{m,n,l}(f,g)(x)|\lesssim
\frac{1}{2^n}\,\Big(\frac{C_{\g}}{2^{m-n}}\Big)^l\,\left(\int\limits_{\{\frac{1}{4} <|t|<1 \}} M^{(2^m t)}f(x)\,M^{(2^n)}_{\g_t} g(x)\,dt\right) \,,
\end{equation*}
and respectively
\begin{equation}\label{A2}
\|\Lambda^{\tilde{A}}_{m,n}(f,g)\|_r\lesssim_{\gamma,p,q} \frac{1}{2^n}(m_++10)\,(n_{+}+10)\,
\|f\|_p\,\|g\|_{q}.
\end{equation}
Again, a similar reasoning applies to the multiplier $\tilde{B}_{j(\cdot),m,n}$.

Putting together all of the above, one concludes
\begin{align*}
\|M_\Gamma^{H\not\Delta}(f,g)\|_r &\lesssim_{\gamma,p,q}
\sum_{\substack{(m,n)\in \Z^2\setminus((\Z_-)^2\cup\Delta)\\ m-n\geq C_{\g}}}
\frac{1}{2^m}(m_++10)\,(n_{+}+10)\,\|f\|_p\,\|g\|_q\\
&\hspace{1cm} +
\sum_{\substack{(m,n)\in \Z^2\setminus((\Z_-)^2\cup\Delta)\\ n-m\geq C_{\g}}}
\frac{1}{2^n}(m_++10)\,(n_{+}+10)\,\|f\|_p\,\|g\|_q\\
&\lesssim_{\gamma,p,q} \|f\|_p\,\|g\|_q.
\end{align*}
\end{proof}

%%%%%%%%%%%%%%%%%%%%%%%%%%%%%%%%%%%%%%%%%%%%%%%%%%%%%%%%%%%%%%%%%%%%%%%%%%%%%%%
%%%%%%%%%%%%%%%%%%%%%%%%%%%%%%%%%%%%%%%%%%%%%%%%%%%%%%%%%%%%%%%%%%%%%%%%%%%%%%%
%%%% subsection - High frequency terms
\subsection{High frequency term close to diagonal}
%% Lambda operators

In this section we consider the last and most relevant component of $m_{j(x)}$, that is, the term $m_{j(x)}^{H\Delta}$.

Our goal for the remaining part of the paper will be to prove the following

\begin{theorem}\label{high_frequency_diagonal theorem}
Set
 \begin{equation}\label{operator_close}
 \M^{H\Delta}(f,g)(x):=\int_\R \int_\R \hat{f}(\xi)\hat{g}(\eta)m_{j(x)}^{H\Delta}(\xi,\eta)e^{i\xi x}e^{i\eta
x}d\xi d\eta.
 \end{equation}
Then for any $p,\,q,\, r$ satisfying \eqref{exponents10} and \eqref{exponents20} we have
\begin{equation}\label{close_bound}
 \|\M_\Gamma^{H\Delta}(f,g)\|_r\lesssim_{\gamma,p,q}\|f\|_p\|g\|_{q}.
\end{equation}
\end{theorem}

In this subsection we will describe the strategy of reducing our Theorem \ref{high_frequency_diagonal theorem} to two intermediate results.

We start by noticing that in \eqref{high-close} it is enough to only consider the case $m=n.$ Also, wlog we assume that our $t$-integration encoded in the expression of $m_{j(x)}^{H\Delta}$ is performed over $\R_+$.

\begin{observation}\label{identif} Recalling Remark 2 in Section \ref{MIR}, since the mean zero condition in \eqref{mzh} plays no role in the regions where the phase of the multiplier is highly oscillatory, we are justified from now on to identify our operators $T_{j,m}$ and $B_{j,m}$ introduced below with the corresponding ones defined in \cite{lie2015boundedness}. Consequently, one can transfer with no modifications the theorems regarding $T_{j,m}$ and $B_{j,m}$ in \cite{lie2015boundedness} to our current setting.
\end{observation}

Using now the notation from \cite[Section 3]{lie2015triangle} and after some elaborate technicalities, one can prove that the study of the operator with multiplier $m_{j(x)}^{H\Delta}$ can be reduced to the study of the bilinear operator $T(f,g)$ defined by
\begin{equation}\label{T}
 T(f,g)(x):=\sum_{m\in\N}T_{j(x),m}(f,g)(x)\,.
\end{equation}
Here, for each $j\in\Z$ and $m\in\N$, we define
\begin{equation}\label{Tjm}
T_{j,m}(f,g)(x):=\int_\R\int_\R\hat{f}(\xi)\hat{g}(\eta)v_{j,m}(\xi,\eta)e^{i\xi x}e^{i\eta x} d\xi d\eta
\end{equation}
and
\begin{equation}\label{vjm}
v_{j,m}(\xi,\eta):=2^{-\frac{m}{2}}e^{i\,\varphi_{\xi,\eta}(t_c)} \zeta\Bigg(\frac{\gamma'(2^{-j})\eta}{2^{m+j}},
\frac{\frac{\xi}{2^{m+j}}}{\frac{\gamma'(2^{-j})\eta}{2^{m+j}}}\Bigg) \phi\Big(\frac{\xi}{2^{m+j}}\Big) \phi\Big(\frac{\gamma(2^{-j})\eta}{2^{m+j}}\Big)
\end{equation}
where we have
\begin{itemize}
 \item the phase of the multiplier - recall \eqref{phmultiplier} - is defined as
 \begin{equation}
\varphi_{\xi,\eta}(t):=-\frac{\xi}{2^j}t+\eta\gamma\big(\frac{t}{2^j}\big).
 \end{equation}
 \item For $\xi,\eta$ fixed, $\varphi_{\xi,\eta}$ has a unique critical point $t_c=t_c(j,\xi,\eta) \in [2^{-k(\gamma)},2^{k(\gamma)}]$, where $k(\gamma)\in\N$ depends only on $\gamma$.
 \item $\zeta:[\frac{1}{10},10]\times[2^{-k(\gamma)},2^{k(\gamma)}]\to\R$ satisfies $\|\zeta\|_{N-3}\lesssim  1$ for some $N\ge7$.\footnote{The regularity index $N$ here can be lowered but we will not detail this fact here.}
\end{itemize}

Also, from the properties of the class $\NF$, wlog we can assume that
\begin{equation}\label{lim}
 \lim\limits_{\substack{t\to0\\ t\not=0}}\gamma'(t)=0\quad \text{ and } \quad \lim\limits_{t\to\infty}\gamma'(t)=\infty.
\end{equation}

Now it turns out that in formulas \eqref{Tjm} and \eqref{vjm} one can replace the function $\zeta$ by $1$. In order to clarify this point and make transparent the parallelism with the reasonings in \cite{lie2015boundedness}, we first need to recall some of the notations that we used in \cite{lie2015boundedness}. Thus, letting $\Psi_\eta(\xi)=-\varphi_{\xi,\eta}(t_c)$, we have

\begin{itemize}
 \item For $j>0$ one sets
\begin{equation}\label{BTtr1}
B_{j,m}(f(\cdot),g(\cdot))(x):=[\gamma'(2^{-j})]^\frac{1}{2}T_{j,m} \Big(f(2^{m+j}\cdot),g\Big(\frac{2^{m+j}}{\gamma'(2^{-j})} \cdot\Big)\Big)\Big(\frac{x\gamma'(2^{-j})}{2^{m+j}}\Big),
\end{equation}
that is
\[
B_{j,m}(f,g)(x)=2^{-\frac{m}{2}}[\gamma'(2^{-j})]^\frac{1}{2} \int_\R \int_\R \hat{f}(\xi)\hat{g}(\eta)e^{i(\gamma'(2^{-j})\xi+\eta)x}e^{-i 2^m 2^{j}\Psi_{\frac{\eta}{\gamma'(2^{-j})}}(\xi)}\zeta(\eta,\frac{\xi}{\eta})\phi(\xi)\phi(\eta)d\xi d\eta.
\]
 \item For $j\le 0$ one sets
\begin{equation}\label{BTtr2}
B_{j,m}(f(\cdot),g(\cdot))(x):=[\gamma'(2^{-j})]^{-\frac{1}{2}}T_{j,m} \Big(f(2^{m+j}\cdot),g\Big(\frac{2^{m+j}}{\gamma'(2^{-j})} \cdot\Big)\Big)\Big(\frac{x}{2^{m+j}}\Big),
\end{equation}
that is
\[
B_{j,m}(f,g)(x)=2^{-\frac{m}{2}}[\gamma'(2^{-j})]^{-\frac{1}{2}} \int_\R \int_\R
\hat{f}(\xi)\hat{g}(\eta)e^{i(\xi+\frac{\eta}{\gamma'(2^{-j})})x} e^{-i 2^m 2^{j}\Psi_{\frac{\eta}{\gamma'(2^{-j})}}(\xi)}
\zeta(\eta,\frac{\xi}{\eta})\phi(\xi)\phi(\eta)d\xi d\eta.
\]
\end{itemize}
In what follows we will assume $j>0$, as the other case $j < 0$ can be treated in a similar way.

From the definition of $\gamma\in\NF$ we define for $x\in\R$
\[
 R(u):=\int_1^u r(s)ds \quad\text\quad R_{j}(u)=\int_1^u r_{j}(s)ds,
\]
then
\[
 2^{j}\Psi_{\frac{\eta}{\gamma'(2^{-j})}}(\xi)=\eta R(\frac{\xi}{\eta})+\eta R_{j}(\frac{\xi}{\eta}).
\]
From the properties of $\g\in\NF$ we know that $\lim_{|j|\rightarrow\infty} \|R_j\|_{C^N}=0$. Thus, by properly choosing $j_0\in\N$ in \eqref{threshold} (based on the properties of $\g\in\NF$), one can assume wlog that one has the pointwise estimate $|R_j|\leq \frac{1}{C_{\g}}|R|$ for some large $C_{\g}>>1$. Consequently, $R_j$ behaves as an error term relative to $R$, and thus, for notational simplicity,  we will discard $R_j$ in what follows.

Thus, we have
\begin{equation}\label{bjxm}
B_{j,m}(f,g)(x):=2^{-\frac{m}{2}}[\gamma'(2^{-j})]^\frac{1}{2} \int_\R
\int_\R \hat{f}(\xi)\hat{g}(\eta)e^{i(\gamma'(2^{-j})\xi+\eta)x}e^{-i 2^m \eta
R(\frac{\xi}{\eta})}\zeta(\eta,\frac{\xi}{\eta}
)\phi(\xi)\phi(\eta)d\xi d\eta.
\end{equation}

With these we are now ready to state the following

\begin{observation}
As in \cite[Section 5]{lie2015boundedness}, one can show that the function
$\zeta$ above can be replaced by the constant function $1$. This brings a series of simplifications especially when dealing later with the situation $p\not=2$. Instead of following the argument in  \cite[Section 5]{lie2015boundedness}, we present here a much simpler approach: the secret lies in changing the perspective and focusing on the function
\[
\varrho(\xi,\eta):=\zeta(\eta,\frac{\xi}{\eta})\,.
\]
Indeed, by doing this, one can perform a double Fourier series development on $\varrho$ and notice that the linear complex exponentials will preserve the curvature of the phase given by $2^m\eta R(\frac{\xi}{\eta})$; in contrast with this, in the original argument focusing on $\zeta(\eta,\frac{\xi}{\eta})$, after the double Fourier series argument one had to work extra in order to deal with expressions of the form $\{\frac{\xi}{\eta}n_2\}_{n_2\in\Z}$.

Returning now to the above definition of $ \varrho$ we notice that $\varrho:[\frac{2^{-k(\gamma)}}{10},10\cdot 2^{k(\gamma)}]
\times[\frac{1}{10},10] \to\R$ satisfies $\|\varrho\|_{C^{N-3}}\lesssim_\gamma
1$. This last property follows from the fact that both $\xi$
and $\eta$ are away from $0$.

We can now assume without loss of generality that $\varrho$ is compactly supported on $[2\pi,4\pi]\times[2\pi,4\pi]$. Regarding now $\varrho$ as a $2\pi$-periodic function on $\R^2$, we represent it as a multiple Fourier series:
\[
\varrho(\xi,\eta)=\sum_{n_1,n_2\in\Z}c_{n_1,n_2}e^{i n_1 \xi}e^{i n_2 \eta}.
\]

From the hypothesis that $\|\varrho\|_{C^{N-3}}\lesssim_\gamma 1$ with $N\geq 7$, we have
\begin{equation}\label{z-coefficients}
|c_{n_1,n_2}|\lesssim_\gamma \frac{1}{(1+|n_1|+|n_2|)^4}.
\end{equation}
Thus, for $j>0$, it follows that
\begin{equation}\label{B-Fourier}
B_{j,m}(f,g)(x)=\sum_{n_1,n_2\in\Z}c_{n_1,n_2}B_{j,m}^{n_1,n_2}(f,g)(x)
\end{equation}
with
\begin{equation}\label{BFourier1}
B_{j,m}^{n_1,n_2}(f,g)(x):=2^{-\frac{m}{2}}[\gamma'(2^{-j})]^\frac{1}{2}
\int_\R \int_\R \hat{f}(\xi)\hat{g}(\eta)e^{i(\gamma'(2^{-j})\xi+\eta)x}e^{-i
	2^m \eta R(\frac{\xi}{\eta})}e^{i\xi n_1}e^{i\eta n_2}\phi(\xi)\phi(\eta)d\xi d\eta.
\end{equation}
At this point we make the following simple observation: $\|B_{j,m}^{0,0}\|_{L^2\times L^2 \to L^1}=\|B_{j,m}^{n_1,n_2}\|_{L^2\times L^2 \to L^1}$ for any $n_1,\,n_2\in \Z$ since
the factors $e^{i\xi n_1}$ and $e^{i\eta n_2}$ can be absorbed into the functions $\hat f$ and $\hat g$ in \eqref{BFourier1} without changing their corresponding $L^2$-norms.

Consequently, since \eqref{z-coefficients} implies the absolute convergence of the sum $$\sum_{n_1,n_2}|c_{n_1,n_2}|\,\|B_{j,m}^{n_1,n_2}\|_{L^2\times L^2 \to L^1}=\|B_{j,m}^{0,0}\|_{L^2\times L^2 \to L^1}\,\sum_{n_1,n_2}|c_{n_1,n_2}|\,,$$
one realizes that the boundedness of each of $T_{j,m}$ can be thought as equivalent with the corresponding boundedness of $B_{j,m}^{0,0}$. Therefore, for notational simplicity, we redenote $B_{j,m}$ as $B_{j,m}^{0,0}$ and set $T_{j,m}$ as the correspondent operator associated with the newly defined $B_{j,m}$.
\end{observation}

Given the observation above, we will only focus our attention on
\begin{equation}\label{nbjxm}
 B_{j,m}(f,g)(x)=2^{-\frac{m}{2}}[\gamma'(2^{-j})]^\frac{1}{2} \int_\R \int_\R \hat{f}(\xi)\hat{g}(\eta)e^{i(\gamma'(2^{-j})\xi+\eta)x} e^{-i 2^m\eta R(\frac{\xi}{\eta})}\phi(\xi)\phi(\eta)d\xi d\eta\,,
\end{equation}
or equivalently, on the corresponding operator $T_{j,m}$ obtained from $B_{j,m}$ via \eqref{BTtr1} (and \eqref{BTtr2} respectively).

Finally, we record the following key relation:
\begin{equation}\label{bt}
\|B_{j,m}\|_{L^2\times L^2 \to L^1}=\|T_{j,m}\|_{L^2\times L^2 \to L^1}\:.
\end{equation}

$\newline$
\noindent\underline{\textsf{Philosophy of our proof}}
$\newline$

\emph{Inspired by \cite{lie2015triangle}, our intention is to show that even in the variable case, the operator
$$T_{j(x),m}\:\:(\textrm{or equivalently}\:\:B_{j(x),m})\,$$
obeys similar decay bounds with $T_{j,m}$ which can be extracted from the corresponding bounds for the $j(x)=j$ constant case. In other words, one can identify a unified approach that deals simultaneously with both the bilinear Hilbert transform and the maximal operator along non-flat curves.}

\begin{figure}[h]
	\includegraphics[scale=.47]{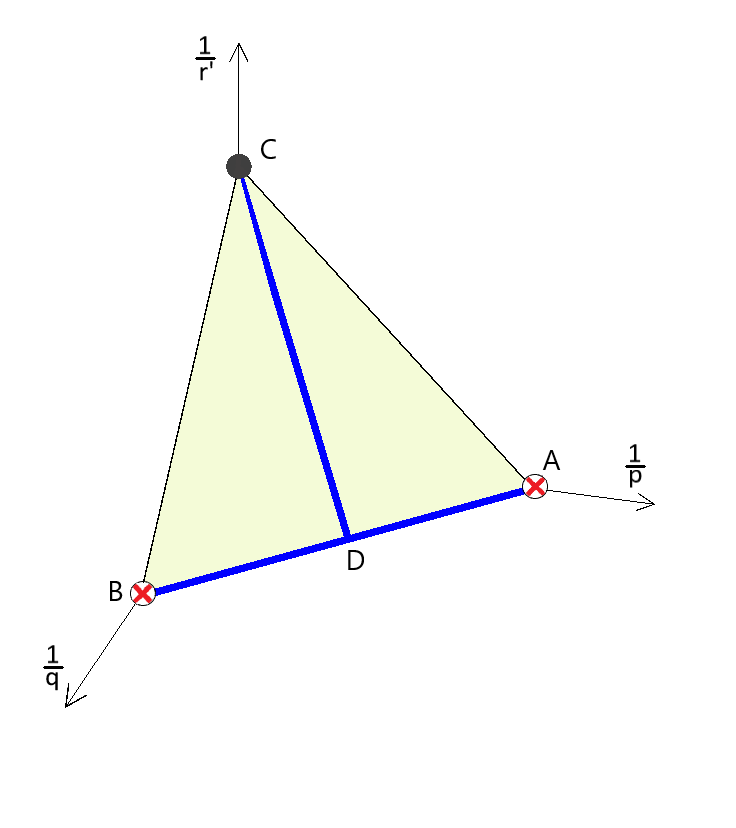}
	\caption{\textbf{Maximal boundedness range for the Bilinear Maximal Function $M_\Gamma$.} Our Main Theorem states that our operator  $M_\Gamma:L^p(\R)\times L^q(\R)\to L^{r}(\R)$ is bounded for all triples $(\frac{1}{p},\frac{1}{q},\frac{1}{r'})$ in the region $\overline{\text{int}(\triangle ABC)}\setminus (\{A\}\cup\{B\})$. This result is sharp.}
	\label{triangle_fig}
\end{figure}

$\newline$
\noindent\underline{\textsf{Main strategy}}
$\newline$

Recall now the definition of our main operator $T(f,g)$ in \eqref{T}. Our proof will be decomposed into two main parts each split in several stages - see Figure \ref{triangle_fig}:
$\medskip$
\begin{itemize}
 \item First part - \underline{boundedness properties of $T$ on $\text{int}(\triangle ABC)\cup (AB)$} - via $m-$decay bounds for $T_{j(x),m}$:
\begin{itemize}
\item in the first stage we provide $m-$decay bounds for $T_{j(x),m}$ on the edge $(AB)$;

\item in the second stage we provide $m-$decay bounds for $T_{j(x),m}$ on the edge $(CD)$;
\end{itemize}

Then our result holds in $\text{int}(\triangle ABC)$ by applying interpolation.
$\medskip$

\item Second part - \underline{boundedness properties of $T$ on $\partial(\triangle ABC)$}:

\begin{itemize}
\item for this situation, in the first stage we provide bounds on the edge $(AC]$;

\item in the second stage we provide bounds on the edge $(BC)$;

\item in the third stage we prove the unboundedness of our operator at the remaining vertices $\{A\}$ and $\{B\}$.
\end{itemize}
\end{itemize}

\section{Boundedness properties of $T$ on $int(\triangle ABC)\cup (AB)$}\label{sec:edges}

In this section, we focus on the boundedness of our operator $T_{j(x),m}(f,g)$ on the edges $(AB)$ and $(CD)$. This together with interpolation imply the boundedness of the our main component operator $T$ for all triples $(\frac{1}{p},\frac{1}{q},\frac{1}{r'})$ within the interior of $\triangle ABC$ together with the edge $(AB)$, \emph{i.e.}, for all $p,\, q,\, r$ satisfying $\frac{1}{p}+\frac{1}{q}=\frac{1}{r}$ with $1<p,\, q<\infty$, and $1\le r<\infty$.

\subsection{Bounds on the edge $(AB)$: $\frac{1}{p}+\frac{1}{q}=1$ with $1<p<\infty$}\label{bound:AB}
$\newline$

Appealing to Observation \ref{identif} and \eqref{bt}, we will transfer in our context the key result in \cite{lie2015boundedness}:

\begin{theorem}{\cite[Theorem 3]{lie2015boundedness}}\label{smallk2}
 There exists $\epsilon\in(0,1)$ such that, for any $j\in \Z$ and $m\in\N$, one has
\begin{equation}\label{l2c}
 \|T_{j,m}(f,g)\|_{L^1}\lesssim_{\gamma} 2^{-\epsilon m}\|f\|_2\,\|g\|_{2}.
\end{equation}
\end{theorem}

We claim that one can get an extension of the above result to the \emph{variable} case, that is

\begin{theorem}\label{smallk20}
 There exists $\epsilon\in(0,1)$ such that, for any $m\in\N$, one has
\begin{equation}\label{l2c1}
 \|T_{j(x),m}(f,g)(x)\|_{L^1(dx)}\lesssim_{\gamma} 2^{-\epsilon m}\|f\|_2\,\|g\|_{2}.
\end{equation}
\end{theorem}

For the case when input functions are not in $L^2$, we get inspired by the route presented in \cite{lie2015triangle}, and prove (at first) the following tame bounds

\begin{theorem}\label{lpsmgr}
For any $m\in\N$ and \footnote{Throughout this subsection $q=p'$.}  $p,\,p'$ satisfying $\frac{1}{p}+\frac{1}{p'}=1$ with $1<p<\infty$, one has
\begin{equation}\label{l2cp}
 \|T_{j(x),m}(f,g)(x)\|_{L^1(dx)}\lesssim_{\gamma,p} m\,\|f\|_p\,\|g\|_{p'}.
\end{equation}
\end{theorem}

Notice now that for the case $r=1$, our Theorem \ref{high_frequency_diagonal theorem} follows from Theorems \ref{smallk20} and \ref{lpsmgr} above via interpolation and geometric summation in the $m-$parameter.

$\newline$
\subsubsection{\textbf{The proofs of Theorems \ref{smallk20} and \ref{lpsmgr}}}
$\newline$

Before starting our journey, for $m\in\N$, we introduce the following operator:
\begin{equation}\label{Tm}
 T_m(f,g)(x):=\sum_{j\in\Z}|T_{j,m}(f,g)(x)|\:.
\end{equation}
We observe that one trivially has
\begin{equation}\label{Tm1}
 |T_{j(x),m}(f,g)(x)|\leq T_m(f,g)(x)\:.
\end{equation}

We thus deduce that both Theorems \ref{smallk20} and \ref{lpsmgr} will be now direct consequences of the following

\begin{theorem}\label{summ}

With the notations in \eqref{T}, \eqref{Tjm} and \eqref{Tm}, the following hold:

\begin{enumerate}
\item There exists $\epsilon\in(0,1)$ such that
\begin{equation}\label{l2c2}
 \|T_{m}(f,g)\|_{L^1}\lesssim_{\gamma} 2^{-\epsilon m}\|f\|_2\,\|g\|_{2}.
\end{equation}

\item For $p,\,p'$ satisfying $\frac{1}{p}+\frac{1}{p'}=1$ with $1<p<\infty$, one has
\begin{equation}\label{l2cpp}
 \|T_{m}(f,g)\|_{L^1}\lesssim_{\gamma,p} m\,\|f\|_p\,\|g\|_{p'}.
\end{equation}
\end{enumerate}
\end{theorem}

\subsubsection{\bf{The proof of \eqref{l2c2}}}

Using the following standard notation
$\phi_{m+j}(\xi):=\phi\Big(\frac{\xi}{2^{m+j}}\Big)$ and $\phi_{\gamma,m,j}(\eta):= \phi\Big(\frac{\gamma'(2^{-j})\eta}{2^{m+j}}\Big)$ we rewrite \eqref{l2c} in  Theorem \ref{smallk2} as
\begin{equation}\label{rewl2c}
\|T_{j,m}(f,g)\|_{L^1}\lesssim_{\gamma} 2^{-\epsilon m}\,\|f*\check{\phi}_{m+j}\|_2\,\|g*\check{\phi}_{\gamma,m,j}\|_{2}.
\end{equation}
Using this together with \eqref{Tm} we deduce that
$$\|T_{m}(f,g)\|_{L^1}\leq \sum_{j\in\Z}\|T_{j,m}(f,g)\|_{L^1}\lesssim_{\gamma} 2^{-\epsilon
m}\,\sum_{j\in\Z}\|f*\check{\phi}_{m+j}\|_2\,\|g*\check{\phi}_{\gamma,m,j}\|_{2}$$
which, via an application of Cauchy-Schwarz, further gives
$$\lesssim 2^{-\epsilon
m}\,\left(\sum_{j\in\Z}\|f*\check{\phi}_{m+j}\|_2^2\right)^\frac{1}{2}\,
\left(\sum_{j\in\Z}\|g*\check{\phi}_{\gamma,m,j}\|_{2}^2\right)^\frac{1}{2}
\lesssim 2^{-\epsilon m}\,\|f\|_2\,\|g\|_2\,,$$
where in the last inequality we used that the supports of the functions $\{\phi_{m+j}\}_{\{j\in\Z\}}$ and  respectively $\{\phi_{\gamma,m,j}\}_{\{j\in\Z\}}$ are almost disjoint, with the latter being a direct consequence of the properties of the curve $\gamma$ (see the property ``smoothness, no critical points, variation near $0$ and $\infty$" in the definition of $\NF$ in \cite{lie2015boundedness}).

\subsubsection{\bf{The proof of \eqref{l2cpp}}} For this point, of fundamental importance is the approach in \cite{lie2015triangle}.

Fix $p,\,p'$ as in our hypothesis and take $f\in L^p$ and $g\in L^{p'}$. In the spirit of \cite[Section 3]{lie2015triangle}, for $h\in L^{\infty}$ (we will assume from now on wlog that $\|h\|_{\infty}=1$ and $h$ is positive), we define\footnote{It is important to notice here that as opposed to the similar object defined in \cite{lie2015triangle}, in our context, in definition \eqref{Ljm} below, the operator $T_{j,m}(f,g)$ is taken with \emph{absolute values} thus making $\Lambda_{j,m}(f,g,h)$ a sublinear ``form".} for each $j\in\Z$ and $m\in\N$
\begin{equation}\label{Ljm}
 \Lambda_{j,m}(f,g,h):=\int_\R |T_{j,m}(f,g)(x)|\, h(x)\,dx.
\end{equation}

For $m\in\N$, we define
\begin{equation}\label{Lm}
 \Lambda_m(f,g,h):=\Lambda_m^+(f,g,h)+\Lambda_m^-(f,g,h),
\end{equation}
where
\begin{equation}\label{Lambda_m_pm}
\Lambda_m^+(f,g,h):=\sum_{j\in\N} \Lambda_{j,m}(f,g,h)\quad \text{ and } \quad \Lambda_m^-(f,g,h):=\sum_{j\in\Z\setminus\N} \Lambda_{j,m}(f,g,h).
\end{equation}
Since our $\Lambda's$ involve absolute values inside the integral expressions, in order to be able to use the techniques in \cite{lie2015triangle}, we will first apply a linearization procedure. Thus, for a suitable sequence of $L^{\infty}-$functions $\{\ep_{j,m}\}$ with the property $|\ep_{j,m}(x)|=1$ a.e $x\in\R$, we will re-write
\begin{equation}\label{Lambdare}
\Lambda_m^+(f,g,h):=\sum_{j\in\N} \int_{\R} T_{j,m}(f,g)(x)\, \ep_{j,m}(x)\,h(x)\,dx,\,
\end{equation}
with the obvious correspondence for $\Lambda_m^-(f,g,h)$.

Deduce now that \eqref{l2cpp} is in fact equivalent with
\begin{equation}\label{l2cppg}
 \left|\Lambda_m^{\pm}(f,g,h)\right|\lesssim_{\gamma,p} m\,\|f\|_p\,\|g\|_{p'}\,\|h\|_{\infty}.
\end{equation}
In what follows we will only focus on the $+$ component since the $-$ component can be treated in a similar fashion.

We will split our discussion in two sub-cases $\newline$

\noindent\textbf{Case 1.} $1<p\leq 2$
$\newline$

With the notations\footnote{We will only recall here that $\phi_{j,p_0}(\eta):= \phi(\frac{\g'(2^{-j})\,\eta}{2^j}-p_0)
$ and $\psi_{m,p_0,j}(\xi):= 2^{-\frac{m}{2}}\,e^{- i p_0 R(\frac{\xi}{2^j\,p_0})}\,\phi(\frac{\xi}{2^{m+j}})$ where $j\in\N$ and $p_0\in [2^m,\, 2^{m+1})\cap \N$ with $m\in\N$.} in \cite{lie2015triangle}, following the first part of the argument provided for the proof of Proposition 4.2 in Section 4 of \cite{lie2015triangle}, we have
$$|\ll_{m}^{+}(f,g,h)|\lesssim$$
$$\int_{\R}\left( \sum_{j\in\N}|(f*\check{\phi}_{m+j})(y)|^2\right)^{\frac{1}{2}}\:
\left( \sum_{j\in\N} \left|M\left(\sum_{p_0=2^m}^{2^{m+1}}|(g*\check{\phi}_{j,p_0})\,((\ep_{j,m} h)*\check{\phi}_{j,p_0})|\right)(y)\right|^2\right)^{\frac{1}{2}}\,dy$$
$$\lesssim \left\| \left( \sum_{j\in\N}|f*\check{\phi}_{m+j}|^2\right)^{\frac{1}{2}}\right\|_{p}\,
\left\|\left( \sum_{j\in\N} \left|M\left(\sum_{p_0=2^m}^{2^{m+1}}|(g*\check{\phi}_{j,p_0})\,((\ep_{j,m}\, h)*\check{\phi}_{j,p_0})|\right)\right|^2\right)^{\frac{1}{2}}\right\|_{p'}$$
$$\lesssim \|f\|_{p}\,\left\|\left( \sum_{j\in\N}\left(\sum_{p_0=2^m}^{2^{m+1}}|(g*\check{\phi}_{j,p_0})\,((\ep_{j,m}\, h)*\check{\phi}_{j,p_0})|\right)^2\right)^{\frac{1}{2}}\right\|_{p'}\:,$$
where for the last relation we used standard Littlewood-Paley theory (for providing bounds on the square function for $f$) and Fefferman-Stein's inequality (\cite{fefferman1971some}) (for the term involving the functions $g$ and $h$).

Using now that
\begin{equation}\label{conv}
((\ep_{j,m}\, h)*\check{\phi}_{j,p_0})(x)=\int_{\R} (\ep_{j,m}\, h)(x-\frac{\g'(2^{-j})}{2^j}\,y)\,\check{\phi}(y)\,e^{i\,p_0\,y}\,dy\:,
\end{equation}
we deduce that
\begin{equation}\label{L2infcont}
\sum_{p_0=2^m}^{2^{m+1}}|((\ep_{j,m}\, h)*\check{\phi}_{j,p_0})|^2\lesssim_{\g} \|h\|_{\infty}^2\,.
\end{equation}

Thus, applying Cauchy-Schwarz inequality, we get
\begin{equation}\label{CS}
\left(\sum_{p_0=2^m}^{2^{m+1}}|(g*\check{\phi}_{j,p_0})\,((\ep_{j,m}\, h)*\check{\phi}_{j,p_0})|\right)^2\lesssim_{\g} \|h\|_{\infty}^2\,\sum_{p_0=2^m}^{2^{m+1}}|(g*\check{\phi}_{j,p_0})|^2\,,
\end{equation}
from which we deduce that
\begin{equation}\label{L}
|\ll_{m}^{+}(f,g,h)|\lesssim_{\g}  \|f\|_{p}\, \|h\|_{\infty}\,\left\|\left(\sum_{j\in\N}\sum_{p_0=2^m}^{2^{m+1}}|g*\check{\phi}_{j,p_0}|^2\right)^{\frac{1}{2}}\right\|_{p'}\:.
\end{equation}

Finally, since $p'\geq2$, we are allowed to apply Rubio de Francia's inequality (\cite{rubio1985littlewood}):
\begin{equation}\label{RF}
\left\|\left( \sum_{j\in\N}\sum_{p_0=2^m}^{2^{m+1}}|g*\check{\phi}_{j,p_0}|^2\right)^{\frac{1}{2}}\right\|_{p'}\lesssim_{\g,p'} \|g\|_{p'}\:.
\end{equation}
Putting now together \eqref{CS}, \eqref{L}, \eqref{RF} we conclude that \eqref{l2cppg} holds.
$\newline$

\noindent\textbf{Case 2.} $2<p<\infty$
$\newline$

In this second case, we follow part of the argument inside the proof of Proposition 4.3 in Section 4 of \cite{lie2015triangle}.

Indeed, by applying Cauchy-Schwarz and then H\"older's inequality, we have
\beq\label{key2}
\eeq
$$|\ll_{m}^{+}(f,g,h)|\leq$$
$$\sum_{j\in\N}\int_{\R}
\,\left(\sum_{p_0=2^m}^{2^{m+1}}|(f*\check{\phi}_{m+j}*\check{\psi}_{m,p_0,j})(x)|^2\right)^{\frac{1}{2}}\times
\left(\sum_{p_0=2^m}^{2^{m+1}}|(g*\check{\phi}_{j,p_0})(x)\,((\ep_{j,m}\, h)*\check{\phi}_{j,p_0})(x)|^{2}\right)^{\frac{1}{2}}\,dx$$
$$\leq\left\|\,\left(\sum_{j\in\N}\sum_{p_0=2^m}^{2^{m+1}}|(f*\check{\phi}_{m+j}*\check{\psi}_{m,p_0,j})|^2\right)^{\frac{1}{2}}\right\|_{p}
\left\|\left(\sum_{j\in\N}\sum_{p_0=2^m}^{2^{m+1}}|(g*\check{\phi}_{j,p_0})\,((\ep_{j,m}\, h)*\check{\phi}_{j,p_0})|^{2}\right)^{\frac{1}{2}}\right\|_{p'}\:.$$

Now the content of Proposition 4.4 in \cite{lie2015triangle} is the statement that for $p\geq2$ one has
\begin{equation}\label{Cf}
\left\|\,\left(\sum_{j\in\N}\sum_{p_0=2^m}^{2^{m+1}}|f*\check{\phi}_{m+j}*\check{\psi}_{m,p_0,j}|^2\right)^{\frac{1}{2}}\right\|_{p}\lesssim_{\g,p} m^{\frac{2}{p'}-1}\, \|f\|_{p}\;,
\end{equation}
while from the Lemma 4.9 in \cite{lie2015triangle} we deduce that
\begin{equation}\label{claim}
\sum_{p_0=2^m}^{2^{m+1}}|(g*\check{\phi}_{j,p_0})(x)\,((\ep_{j,m}\, h)*\check{\phi}_{j,p_0})(x)|^{2}\lesssim_{\g} \|h\|_{\infty}^2\, M(g*\check{\phi}_{\g,m,j})^2(x)\:,
\end{equation}
where here we recall that $\phi_{\g,m,j}(\eta):=\phi(\frac{\g'(2^{-j})}{2^{j+m}}\,\eta)\;.$

Inserting \eqref{Cf} and \eqref{claim} in \eqref{key2}, we conclude that

\beq\label{key3}
\eeq
$$|\ll_{m}^{+}(f,g,h)|\leq_{\g,p}  m^{\frac{2}{p'}-1}\, \|f\|_{p}\,\|h\|_{\infty}\,
\left\|\left(\sum_{j\in\N} M(g*\check{\phi}_{\g,m,j})^2\right)^{\frac{1}{2}}\right\|_{p'}$$
$$\lesssim_{p} m^{\frac{2}{p'}-1}\, \|f\|_{p}\,\|h\|_{\infty}\,
\left\|\left(\sum_{j\in\N} (g*\check{\phi}_{\g,m,j})^2\right)^{\frac{1}{2}}\right\|_{p'}\lesssim_{\g,p} m^{\frac{2}{p'}-1}\, \|f\|_{p}\,\|g\|_{p'}\,\|h\|_{\infty}\,,$$
where we made use again of the Fefferman-Stein's inequality (\cite{fefferman1971some}) (for the second inequality) and standard Littlewood-Paley theory (for the third inequality).

%%%%%%%%%%%%%%%%%%%%%%%%%%%%%%%%%%%%%%%%%%%%%%%%%%%%%%%%%%%%%%%%%%%%%%%%%%%%%%%
%%%%%%%%%%%%%%%%%%%%%%%%%%%%%%%%%%%%%%%%%%%%%%%%%%%%%%%%%%%%%%%%%%%%%%%%%%%%%%%
%%%% subsection - p=q=2r

\subsection{Bounds on the segment $(CD)$: $p=q=2r$ with $1<r<\infty$}
$\newline$

The main result of this subsection is
\begin{theorem}\label{2r}
	Let $1\le r<\infty$. For $m\in\N$, the following holds
	\begin{equation}
	\|T_{j(x),m}(f,g)(x)\|_{L^r(dx)}\lesssim_{\gamma,r}(1+m^{1-\frac{1}{r}})\,2^{-\frac{m}{20r}}\,\|f\|_{2r}\,\|g\|_{2r}.
	\end{equation}
\end{theorem}

In order to prove our Theorem \ref{2r} we will need the following

\begin{proposition}\label{Tjmb}
	Let $p, q, r$ be as in \eqref{exponents10} and \eqref{exponents20}. Then for any $j, m\in\Z$, we have
	\begin{equation}\label{Tjm-pqr}
	\|T_{j,m}(f,g)\|_r\lesssim_{\gamma,p,q}(1+m^{1-\frac{2}{p}})\,2^{-\frac{m}{10}(1-\frac{1}{p*})}\,\|f\|_p\,\|g\|_q.
	\end{equation}
\end{proposition}

\begin{proof}
Using the notation from (\ref{Ljm}) and \cite{lie2015triangle}, we choose $h\in L^{r'}$ such that
\[
\Lambda_{j,m}(f,g,h)= \int_\R T_{j,m}(f,g)(x)h(x)dx.
\]
Therefore
\[
	|\Lambda_{j,m}(f,g,h)|\le %\sum_{j\in\N}
	\int_\R |(f*\check{\phi}_{m+j})(x)\:|M\Bigg(\sum_{p_0=2^m}^{2^{m+1}}|(g*\check{\phi}_{j, p_0})(\cdot) (h*\check{\phi}_{j, p_0})(\cdot)|\Bigg)(x)\,dx,
\]
and
\[
	|\Lambda_{j,m}(f,g,h)|\le %\sum_{j\in\N}
	\int_\R\Bigg(\sum_{p_0=2^m}^{2^{m+1}}|(g*\check{\phi}_{j, p_0})(x)(h*\check{\phi}_{j, p_0})(x)|^2\Bigg)^\frac{1}{2} \Bigg(\sum_{p_0=2^m}^{2^{m+1}}|(f*\check{\phi}_{m+j}*\check{\psi}_{m, p_0, j})(x)|^2\Bigg)^\frac{1}{2}dx.
\]

Hence, by following the second part of the proof of \cite[Proposition  4.2 ]{lie2015triangle} and \cite[Proposition 4.3 ]{lie2015triangle}, we deduce that the results in \cite{lie2015triangle} hold for $\Lambda_{j,m}$. In particular, the analogue of \cite[Theorem 3.3]{lie2015triangle} follows:

\begin{proposition}\label{DeltaM}
Let $1<p<\infty$. Then the following estimates hold
\begin{equation}
Edge\:(AC):\:\:\:\:\:\:\:\:|\Lambda_{j,m}(f,g,h)|\lesssim_{\gamma,p} (1+m^{1-\frac{2}{p}})2^{-\frac{m}{10}(1-\frac{1}{p*})}\|f\|_p\|g\|_\infty \|h\|_{p'},
\end{equation}
and
\begin{equation}
Edge\:(AB):\:\:\:\:\:\:\:\:|\Lambda_{j,m}(f,g,h)|\lesssim_{\gamma,p} (1+m^{1-\frac{2}{p}})2^{-\frac{m}{10}(1-\frac{1}{p*})} \|f\|_p\|g\|_{p'} \|h\|_\infty.
\end{equation}
\end{proposition}

Thus, Proposition \ref{Tjmb} follows from the Proposition \ref{DeltaM} after applying real interpolation.
\end{proof}
\medskip

We are now ready for the following:
\begin{proof}[Proof of Theorem \ref{2r}]
$\newline$
	
	As we did for the proof of \eqref{l2c2}, we write \eqref{Tjm-pqr} as
\begin{equation}
\|T_{j,m}(f,g)\|_r\lesssim_{\gamma}(1+m^{1-\frac{1}{r}})2^{-\frac{m}{20r}}\|f*\check{\phi}_{m+j}\|_{2r}\|g*\check{\phi}_{\gamma,m,j}\|_{2r},
\end{equation}
where here we used that $p=q=2r$ and that $(2r)^*=\frac{2r}{2r-1}$ since $r\ge 1$.

Since
\begin{equation*}
	|T_{j(x),m}(f,g)(x)|\le \Big(\sum_{j\in\Z}|T_{j,m}(f,g)|^r\Big)^\frac{1}{r},
\end{equation*}
it follows from H\"older inequality and Proposition \ref{Tjmb} that
\begin{align}
	\big\|T_{j(x),m}(f,g)(x)\big\|_{L^r(dx)}&\le \Big(\sum_{j\in\Z}\|T_{j,m}(f,g)\|_r^r\Big)^\frac{1}{r}\nonumber\\
	&\le (1+m^{1-\frac{1}{r}})2^{-\frac{m}{20r}} \Big(\sum_{j\in\Z}\|f*\check{\phi}_{m+j}\|_{2r}^r\|g*\check{\phi}_{\gamma,m,j}\|_{2r}^r\Big)^\frac{1}{r}.\label{supTjm}
\end{align}

Notice that for $r\in [1,\infty)$ and for any $f, \phi\in\s(\R)$, we have
\begin{align*}
\Big(\sum_{j\in\Z}\|f*\phi_j\|_{2r}^{2r}\Big)^\frac{1}{2r}&\le \Bigg(\int_\R \sum_{j\in\Z}|(f*\phi_j)(x)|^{2r}dx\Bigg)^\frac{1}{2r}\\
&\le \Bigg(\int_\R \Bigg(\sum_{j\in\Z}|(f*\phi_j)(x)|^2\Bigg)^rdx\Bigg)^\frac{1}{2r} = \Bigg(\int_\R \Bigg\{\bigg(\sum_{j\in\Z}|(f*\phi_j)(x)|^2\bigg)^\frac{1}{2}\Bigg\}^{2r}dx\Bigg)^\frac{1}{2r}\\
&=\Bigg\|\bigg(\sum_{j\in\Z}|f*\phi_j(x)|^2\bigg)^\frac{1}{2}\Bigg\|_{2r}
\lesssim \|f\|_{2r},
\end{align*}
here for the last relation we used standard Littlewood-Paley theory.

By Cauchy Schwartz, the hypothesis on $\g\in \NF$ and the result above, we obtain
\begin{align}
\Big(\sum_{j\in\Z}\|f*\check{\phi}_{m+j}\|_{2r}^r\|g*\check{\phi}_{\gamma,m,j}\|_{2r}^r\Big)^\frac{1}{r}&\le \Big(\sum_{j\in\Z}\|f*\check{\phi}_{m+j}\|_{2r}^{2r}\Big)^\frac{1}{2r}\Big(\sum_{j\in\Z}\|g*\check{\phi}_{\gamma,m,j}\|_{2r}^{2r}\Big)^\frac{1}{2r} \nonumber\\
&\le_{\g,r} \|f\|_{2r}\|g\|_{2r}.\label{fg2r}
\end{align}

Inserting \eqref{fg2r} in \eqref{supTjm} we get
\begin{equation}
\big\|T_{j(x),m}(f,g)(x)\big\|_{L^r(dx)}\le_{\g,r} (1+m^{1-\frac{1}{r}})2^{-\frac{m}{20r}}\|f\|_{2r}\|g\|_{2r},
\end{equation}
as desired.
\end{proof}

\begin{observation}
In order to prove our main Theorem \ref{high_frequency_diagonal theorem} in the interior of $\triangle ABC$ it is in fact not necessary to pursue the boundedness of our operator along the edge $(CD)$. Indeed, one could apply interpolation between the bounds corresponding to the edges $(AB)$ and $(AC)$ instead, with the latter discussed in the next section. However, we wanted to offer a less expected but more interesting alternative approach to the trivial bounds one gets along the segment $(AC)$. This is especially useful in situations in which one deals with similar type operators but for which one does not have a good control on the edges $(BC)$ and $(AC)$.
\end{observation}

\section{Boundedness properties of $T$ on $\partial(\triangle ABC)$}
In this section we will prove positive boundedness results along the edges $(AC]$, $(BC]$ and negative results for the vertices $\{A\}$ and $\{B\}$ (see Figure \ref{triangle_fig}).\footnote{We remind the reader that the boundedness along the edge $(AB)$ was proved in Subsection \ref{bound:AB}.}

\subsection{Bounds on the edge $(AC]$: $1< p\le \infty$, $r=p$ and $q=\infty$}
$\newline$

From the definition \eqref{definition} of $M_\Gamma$ it is trivial to notice that for all $x\in\R$,
\begin{align}
|M_\Gamma(f,g)(x)|&\le \|g\|_\infty\, \sup_{\e>0}\frac{1}{2\e}\int_{-\e}^\e |f(x-t)|dt\nonumber\\
&\le \|g\|_\infty\, Mf(x).\label{1inf1}
\end{align}
Hence, applying the classical results on the Hardy-Littlewood maximal operator we get
\begin{theorem}\label{ginfty}
	For any $1<p\le \infty$, one has
	\begin{equation}
	\|M_\Gamma(f,g)\|_{L^p}\lesssim_p \|f\|_p\,\|g\|_\infty.
	\end{equation}
\end{theorem}
$\smallskip$

\subsection{Bounds on the edge $(BC]$: $1< q\le \infty$, $r=q$ and $p=\infty$}
$\newline$

Our goal in this section is to prove the following:
\begin{theorem}\label{finfty}
	For any $1<q\le \infty$, one has
	\begin{equation}
	\|M_\Gamma(f,g)\|_{L^q}\lesssim_{\g,q} \|f\|_\infty\,\|g\|_q.
	\end{equation}
\end{theorem}

From \eqref{discretizationc}, recalling that $\gamma\in\NF$ and our choice of $j_0$ in \eqref{threshold} - \eqref{threshold1}, we notice that for all $x\in\R$,
\begin{align}
|M_\Gamma(f,g)(x)|&\le \|f\|_\infty\,M_\gamma\, g(x)\label{ginf}.
\end{align}
where here
\begin{equation}
M_\gamma\, g(x):=\sup\limits_{{j\in\Z}\atop{|j|>j_0}}\frac{1}{2^{j+1}}\int\limits_{2^j\le|t|\le2^{j+1}}  |g(x+\gamma(t))|dt\,.
\end{equation}

Thus, Theorem \ref{finfty} follows from the result below
\begin{theorem}\label{Mgamma}
	For any $1<q\le \infty$, one has
	\begin{equation}
	\|M_\gamma(g)\|_{L^q}\lesssim_{\gamma,q} \|g\|_q.
	\end{equation}
\end{theorem}

\begin{proof}
From the properties of $\g$ and the choice of $j_0$ in \eqref{threshold} - \eqref{threshold1}, it is straightforward to check that for any $k\in \Z$ with $|k|>j_0$  and $t\in[2^k,2^{k+1}]$ one has
\begin{equation}\label{gammak1t}
|\gamma(2^{k})|\approx_{\gamma} 2^{k}|\gamma'(t)|\approx_{\gamma}|\gamma(2^{k+1})|\,.
\end{equation}

From \eqref{threshold1}, we deduce that $\gamma$ is strictly monotone and invertible over each of the intervals
$(-\infty,-2^{j_0})$, $(-2^{-j_0},0)$, $(0,2^{-j_0})$  and $(2^{j_0}, \infty)$. Wlog we consider from now on that our entire discussion takes place relative to $J_0:=(2^{j_0}, \infty)$ and that both $\g'$ and $\g''$ are strictly positive on $J_0$. Thus, for $k\in\Z$ with $k>j_0$, one has
\begin{align*}
\frac{1}{2^{k+1}}\int_{2^k}^{2^{k+1}}|g(x+\gamma(t))|dt &=\frac{1}{2^{k+1}}\int_{\gamma([2^k,2^{k+1}])}|g(x+u)|\frac{du}{|\gamma'(\gamma^{-1}(u))|}\\
&\le\frac{1}{2^{k+1}|\gamma'(2^k)|}\int_{\gamma(2^k)}^{\gamma(2^{k+1})}|g(x+u)|du\,.
\end{align*}
Appealing now to \eqref{gammak1t}, we deduce
\begin{equation}\label{ks}
\frac{1}{2^{k+1}}\int_{2^k}^{2^{k+1}}|g(x+\gamma(t))|dt \leq\frac{1}{2^{k+1}|\gamma'(2^k)|}\int_{-|\gamma(2^k)|-|\gamma(2^{k+1})|}^{|\gamma(2^k)|+|\gamma(2^{k+1})|}|g(x-u)|du
\lesssim_{\gamma} Mg(x).
\end{equation}

Analogously,
\begin{align*}
\frac{1}{2^{k+1}}\int_{-2^{k+1}}^{-2^k}|g(x+\gamma(t))|dt&\lesssim_{\gamma} Mg(x).
\end{align*}

Therefore
\begin{equation}
\sup_{{j\in\Z}\atop{|j|>j_0}}\,\frac{1}{2^{j+1}}\int\limits_{2^j\le|t|\le2^{j+1}} |g(x+\gamma(t))|dt\lesssim_{\gamma} Mg(x).
\end{equation}
Using now the standard theory on Hardy Littlewood maximal operator we conclude our Theorem \ref{Mgamma}.
\end{proof}

\medskip
\subsection{Behavior on $\{A\}\cup\{B\}$: triples $p=r=1,\, q=\infty$ and $p=\infty,\, q=r=1$}
$\newline$

In this section we show that our operator\footnote{Throughout this section we return to the original definition of our operator $M_\Gamma(f,g)$ in \eqref{definition}.} is unbounded at the points $\{A\}$ and $\{B\}$ (see Figure \ref{triangle_fig}).

$\newline$
\noindent\textbf{Case 1.} \textsf{Vertex $A$}:
$\newline$

We start by focusing on the vertex $A$, \ie, $p=r=1$ and $q=\infty$. In this case, we show that there are functions $f\in L^1(\R)$, $g\in L^\infty(\R)$, and a map $\gamma\in\NF$ such that $M_\Gamma(f,g)$ is not bounded on $L^1$.

Let $f=\chi_{[-1,1]}$ and $g=\chi_{(-\infty,1)}$ be characteristic functions, and let $\gamma(t)=t^2\in\NF$. Note that $f\in L^1(\R)$, $g\in L^\infty(\R)$ with  $\|f\|_1=2$, and $\|g\|_\infty=1$. Then
\begin{align*}
M_\Gamma(f,g)(x)=\sup_{\e>0}\frac{1}{2\e}\int\limits_{-\e}^\e |f(x-t)g(x-t^2)|dt\ge \frac{1}{2(|x|+1)}\int\limits_{-(|x|+1)}^{|x|+1}f(x-t)g(x-t^2)dt.
\end{align*}
Suppose $x>2$, then
\begin{align*}
M_\Gamma(f,g)(x)\ge \frac{1}{2(x+1)}\int\limits_{x-1}^{x+1}f(x-t)g(x-t^2)dt=\frac{1}{2(x+1)}\int\limits_{x-1}^{x+1}g(x-t^2)dt\gtrsim\frac{1}{x+1}.
\end{align*}
Therefore, $M_\Gamma(f,g)$ is not absolutely integrable.

\begin{observation}
Notice that if $f\in L^1$ and $g\in L^\infty$ we have by \eqref{1inf1}
\begin{align*}
\{x\in\R:M_\Gamma(f,g)(x)>\alpha\}\subseteq \Big\{x\in\R:Mf(x)>\frac{\alpha}{\|g\|_\infty}\Big\},
\end{align*}
for any $\alpha>0$.
Thus, from the classical theory, we get that for any $\alpha>0$ the following holds:
\begin{align*}
|\{x\in\R:M_\Gamma(f,g)(x)>\alpha\}|\lesssim \frac{1}{\alpha}\|f\|_1 \|g\|_\infty\:.
\end{align*}
\end{observation}

$\newline$
\noindent\textbf{Case 2.} \textsf{Vertex $B$}:
$\newline$

In this situation we deal with the vertex $B$, given by $p=\infty,\, q=r=1$. We will show that there are functions $f\in L^\infty(\R)$ and $g\in L^1(\R)$ such that $M_\Gamma(f,g)$ is not bounded on $L^1$ for, say, $\gamma(t)=t^2$.

Let $f$ be the constant function 1, and let $g=\chi_{[-1,1]}\in L^1(\R)$. Then
\begin{align*}
M_\Gamma(f,g)(x)=\sup_{\e>0}\frac{1}{2\epsilon}\int_{-\e}^\e|g(x-t^2)|dt \simeq %\sup_{j\in\Z} \frac{1}{2^j}\int_{2^j}^{2^{j+1}} |g(x-t^2)|dt\\&\le  %
\sup_{\e>0}\frac{1}{\e}\int_\e^{2\e} |g(x-t^2)|dt.
\end{align*}

Notice that for any $\e>0$, we have
\begin{align*}
\frac{1}{\e}\int_\e^{2\e}|g(x-t^2)|dt=\frac{1}{\e}\int_{\e^2}^{4\e^2}\frac{|g(x-u)|}{2\sqrt{u}}du\ge \frac{1}{4\e^2}\int_{\e^2}^{4\e^2}|g(x-u)|du.
\end{align*}

Then, for $x>2$ we have
\begin{align*}
\frac{1}{\sqrt{(x+1)/4}} \int_{\sqrt{(x+1)/4}}^{2\,\sqrt{(x+1)/4}}|g(x-t^2)|dt\ge \frac{1}{x+1} \int\limits_{(x+1)/4}^{4(x+1)/4}  |g(x-u)|du \gtrsim \frac{1}{x+1}.
\end{align*}

Therefore, $M(f,g) \not\in L^1(\R)$.

\end{document}